\documentclass[a4paper,11pt]{amsart}

\usepackage[margin=3.cm]{geometry}

 \usepackage{hyperref} 
\usepackage{lipsum}
\usepackage{amsthm}
\usepackage{amsmath}
\usepackage{amsbsy}
\usepackage{bigints}
\usepackage{amsfonts}%
\usepackage{amssymb}%
\usepackage{graphicx}
\usepackage[utf8]{inputenc}%
\usepackage[T1]{fontenc}%
\usepackage[english]{babel}%
\usepackage[all]{xy}
\usepackage{amsmath, amsthm, amssymb,amsfonts,mathrsfs,amscd,latexsym}
\usepackage{graphicx}
\usepackage{tikz-cd}
\usepackage{mathtools}
\usepackage{hyperref}
\usepackage{graphicx}
\graphicspath{ {./image/} }

\newcommand{\poubelle}[1]{}

\usepackage{wrapfig}

\usepackage[all]{xy}
\usepackage{colortbl}

\usepackage{amsmath,amssymb,graphics,epsfig,color}
\usepackage{enumerate,wasysym,stmaryrd}

\usepackage{color}
\usepackage{fancybox}
\usepackage{colordvi}
\usepackage{multicol}
\usepackage{colordvi}
\usepackage{geometry}
\usepackage{lscape}
\usepackage{amsthm}
\usepackage{eqnarray,amsmath}
\usepackage{array}
\usepackage{booktabs}   
  \newcolumntype{x}[1]{>{\centering\hspace{0pt}}p{#1}}
\def\R{\mathbb{R}}

\def\{{\left\lbrace}
\def\}{\right\rbrace}

\newtheorem{theorem}{Theorem}[section]
\newtheorem{conjecture}{Conjecture}
\newtheorem{corollary}[theorem]{Corollary}
\newtheorem{define}[theorem]{Definition}

\newtheorem{lemma}[theorem]{Lemma}
\newtheorem{prop}[theorem]{Proposition}
\newtheorem{remark}[theorem]{Remark}
\newtheorem{convention}[theorem]{Convention}

\newtheorem{innercustomthm}{Theorem}[section]

\newenvironment{customthm}[1]
  {\renewcommand\theinnercustomthm{#1}\innercustomthm}
  {\endinnercustomthm}
  
\newtheorem{iprop}{Proposition}
\newtheorem{icorollary}{Corollary}

\renewcommand{\v}{\mathrm{v}}
\newcommand{\w}{\mathrm{w}}
\newcommand{\aut}{\mathrm{Aut}}
\newcommand{\red}{\mathrm{red}}

\newcommand{\PL}{\textnormal{P,\textbf{L}}}

\usepackage[utf8]{inputenc}

\title{A Yau-Tian-Donaldson correspondence on a class of toric fibrations}

\author{Simon Jubert}

\begin{document}

\maketitle

\begin{abstract}
We establish a Yau--Tian--Donaldson type correspondence, expressed in terms of a single Delzant polytope, concerning the existence of extremal K\"ahler metrics on a large class of toric fibrations,  introduced by Apostolov--Calderbank--Gauduchon--Tonnesen-Friedman  and called \textit{semi-simple
principal toric fibrations}. We use that an extremal metric on the total
space corresponds to a weighted constant scalar curvature K\"ahler metric (in the sense of Lahdili) on the
corresponding toric fiber in order to obtain an equivalence between the existence of extremal K\"ahler metrics on the total space and a suitable notion of weighted uniform K-stability of the corresponding Delzant polytope. As an application, we show that the projective plane bundle $\mathbb{P}(\mathcal{L}_0\oplus\mathcal{L}_1 \oplus \mathcal{L}_2)$, where $\mathcal{L}_i$ are holomorphic line bundles over an elliptic curve, admits an
extremal metric in every K\"ahler class.
\end{abstract}

\section{Introduction}

\subsection{Motivation}
A central problem in K\"ahler geometry,  proposed by Calabi \cite{EC2} in the 1980's, is to find a canonical K\"ahler metric in a given cohomology class of a compact K\"ahler manifold. Calabi suggested looking for \textit{extremal K\"ahler metrics}, characterized  by the property that the flow of the gradient of the scalar curvature  preserves the complex structure \cite{EC}. Constant scalar curvature K\"ahler (cscK for short) metrics and the much studied K\"ahler–Einstein metrics are particular examples of such metrics.

The existence of an extremal K\"ahler metric in a given K\"ahler class is conjecturally equivalent to a certain notion of \textit{stability} through an extension of the Yau-Tian-Donaldson's (YTD) conjecture, introduced \cite{GGS2, GGS} for the polarized case and \cite{RDJR} for a  general K\"ahler class. This conjecture, its ramification \cite{GGS2, GGS} and extension \cite{VA7, DR, EI, AL, GGS2, GGS,  ZSD} have generated tremendous efforts in K\"ahler geometry and have led to many interesting developments during the last decades.

The question has been settled in some special cases, especially on smooth toric varieties \cite{XC1, XC2, SKD, TH, CL, ZZ} where the relevant stability notion is expressed in terms of the convex affine geometry of the corresponding Delzant polytope,  and is referred to as \emph{uniform K-stability} of the polytope. Other special cases include Fano manifolds (see e.g. \cite{BBJ, CSS, CSS2, CSS3, LTW,  GT}), total spaces of projective line bundles over a cscK base \cite{VA8, VA3} and certain varieties with a large symmetry group \cite{TDD}. In general, though, the YTD conjecture is still open and it is expected that the relevant notion of stability would be the one of  \emph{relative uniform K-stability},  see e.g. \cite{BBJ, RD, GGS2, GGS}.

In  \cite{VA2}, the authors introduced a class of fiber bundles, called \textit{semi-simple rigid toric bundles}, which have toric K\"ahler fibers. They are obtained from \textit{the generalized Calabi construction} \cite{VA2, VA3} involving a product of polarized K\"ahler manifolds, a certain principal torus bundle and a given toric K\"ahler manifold. In this paper, we are interested in the special case of \textit{semi-simple principal toric fibration}. Namely, this is the case when the base is a global product of cscK Hodge manifolds, and there are no blow-downs, see \cite{VA3} or Remark \ref{remark-semi-simple} below. Examples include the total space of the projectivisation of a direct sum of holomorphic line bundles over a compact complex curve, as well as the $\mathbb{P}^1$-bundle constructions over the product of cscK Hodge  manifolds originally used by Calabi \cite{EC2}  and generalized in many subsequent works (see e.g.
 \cite{VA8, DG2, ADH, ADH2, KS, YS, CTF}). On any semi-simple rigid toric fibration, the authors of \cite{VA2} introduced a class of K\"ahler metrics, called \textit{compatible K\"ahler metrics}. A K\"ahler class containing a compatible K\"ahler metric is referred to as a \textit{compatible K\"ahler class}. For any compatible K\"ahler metric on $M$, the momentum map of the toric K\"ahler fiber $(V, \omega_V, J_V, \mathbb{T})$  can be identified with the momentum map of the induced $\mathbb{T}$-action on $(M, J, \omega_M)$, so that  the momentum image of $M$ is the Delzant polytope $P$ of $(V, \omega_V, J_V, \mathbb{T})$.  In \cite{VA3}, the authors made the following conjecture:

\begin{conjecture}[\cite{VA3}]{\label{Conjecture}}

Let $(M,J,\omega,\mathbb{T})$ a  semi-simple rigid toric fibration and $P$ its associated Delzant polytope. Suppose $[\omega]$ is a compatible K\"ahler class. Then the following statements are equivalent:

\begin{enumerate}
    \item $(M,J,[\omega])$ admits an extremal K\"ahler metric;
    \item $(M,J,[\omega])$ admits a compatible extremal K\"ahler metric;
    \item $P$ is weighted K-stable.
\end{enumerate}

\end{conjecture}

\noindent In the third assertion, the notion of \textit{weighted stability} is a weighted version of the notion of K-stability introduced in \cite{SKD} (see \S \ref{subsection-K-stability-and-proper}) for appropriate values of the weight functions, and asks for the positivity of a linear functional defined on the space of convex piece-wise linear functions which are not affine-linear over $P$.

\subsection{Main results} Our purpose in this paper is to solve Conjecture \ref{Conjecture} for semi-simple principal toric fibrations.

\begin{customthm}{1}{\label{theorem1}}
For $M$ a semi-simple principal toric fibration, Conjecture $\ref{Conjecture}$ is true if we replace  condition $(3)$ with the  notion of weighted uniform K-stability, see Definition $\ref{uniform-K-stable}$.
\end{customthm}

\noindent In the above statement, in order to define \emph{uniform} (weighted) K-stability (see Definition \ref{uniform-K-stable}), we use normalized continuous convex functions which are smooth in the interior
of $P$ and the usual $L^1$-norm of $P$. By $C^0$ density and continuity, this is equivalent to the uniform (weighted) K-stability of $P$,  defined in terms of normalized convex piecewise linear functions and the $L^1$-norm.~\footnote{After the submission of the first version of our article on the arXiv, we have been contacted by Yasufumi Nitta who kindly shared with us his manuscript with Shunsuke Saito in which the authors establish, in the case of a polarized toric variety,  the equivalence between various notions of uniform K-stability of $P$. In particular, their result gives a strong evidence and establishes in a certain case the equivalence between the uniform weighted K-stability and and a suitable notion of weighted K-stability of $P$ in Conjecture \ref{Conjecture} $(3)$.}

 We split Theorem \ref{theorem1} in two statements: Theorem \ref{theoremAA} and Theorem \ref{theoremBB} below. Theorem \ref{theoremAA} corresponds to the statement "(1) $\Leftrightarrow$ (2)" in Conjecture \ref{Conjecture}.

\begin{customthm}{2}[Theorem \ref{theoremA}]{\label{theoremAA}}

Let $(M,J, \omega_M, \mathbb{T})$ be a semi-simple principal toric fibration with fiber $(V,J_V, \omega_V, \mathbb{T} )$. Then, the following statements are equivalent:

\begin{enumerate}
    \item there exists an extremal K\"ahler metric in $(M,J,[\omega_M],\mathbb{T})$;
      \item  there exists a compatible extremal K\"ahler metric in $(M,J,[\omega_M],\mathbb{T})$;
    \item  there exists a weighted cscK metric in $(V,J_V,[\omega_V],\mathbb{T})$ for the weights defined in $(\ref{weights})$ below.
\end{enumerate}

\end{customthm}

 In the third assertion, the notion of weighted cscK metric is in the sense of \cite{AL}, see \S\ref{section-scalv} for a precise definition. The equivalence $(2) \Leftrightarrow (3)$, established in \cite{VA2}, is recalled in \S \ref{section-rigidtoric}.  The main idea behind the proof of $(1) \Rightarrow (2)$ is to use that \cite{XC1, XC2, WH} the existence of an extremal K\"ahler metric implies a certain properness condition of the corresponding relative Mabuchi functional (see Theorem \ref{Chen--Cheng-existence} below for a precise statement) and then shows that the continuity path of \cite{XC5} can be made in the subspace of \emph{compatible} K\"ahler metrics in $[\omega_M]$. The deep  results \cite{XC1,XC2, WH} then yield the existence of an extremal K\"ahler metric in $[\omega_M]$ given  by the generalized Calabi construction of \cite{VA2}.

Recently, building on the proof of Theorem \ref{theoremAA}, a similar statement  was established in \cite{VA6} for a larger class of fibrations  associated to   a certain class of principal $\mathbb{T}$-bundles over products of cscK Hodge manifolds, whose fiber is an arbitrary compact K\"ahler manifold containing $\mathbb{T}$ in its reduced automorphism group.

Theorem \ref{theoremBB} below corresponds to the statement "(1) $\Leftrightarrow$ (3)" in Conjecture \ref{Conjecture} and provides a \emph{criterion} for verifying the equivalent conditions of Theorem \ref{theoremAA}, expressed in terms of the Delzant polytope of the fiber and data depending on the topology of $M$ and the compatible K\"ahler class.

\begin{customthm}{3}[Theorem \ref{theorem-B}]{\label{theoremBB}}
Let $(M,J, \omega_M, \mathbb{T})$ be a semi-simple principal toric fibration with fiber $(V,J_V, \omega_V, \mathbb{T})$ and  denote by $P$ its associated Delzant polytope. Then there exists a  weighted cscK metric in $[\omega_V]$ if and only if $P$ is weighted uniformly $K$-stable, for the weights defined in $(\ref{weights})$. In particular, the latter condition is necessary and sufficient for $[\omega_M]$ to admit an extremal K\"ahler metric.
\end{customthm}

\noindent The strategy of proof of the above result consists in considering the extremal K\"ahler metrics on the total space $(M,J,\mathbb{T})$ as weighted $(\v,\w)$-cscK metrics on the corresponding toric fiber $(V,J_V, \omega_V, \mathbb{T})$ via Theorem \ref{theoremAA}. 
We then use  the Abreu--Guillemin formalism and a weighted adaptation of the  results in  \cite{CLS, SKD, ZZ} to establish the equivalence on $(V, \omega_V,  J_V, \mathbb{T})$: in one direction, namely showing that the existence implies that polytope is weighted uniformly K-stable, the argument follows from a straightforward modification of the result in \cite{CLS},  which appears in \cite{LLS2}. To show the other direction, we build on \cite{SKD, ZZ} to obtain in Proposition \ref{stable-equivaut-energy-propr-v} that  the uniform weighted K-stability of the polytope implies a certain notion of coercivity of the weighted Mabuchi energy. We then show that the latter implies the properness  of the Mabuchi energy of $M$, and we finally conclude by invoking  again \cite{XC1,XC2, WH}.

Finally, we will be interested in a certain class of almost K\"ahler metrics on a toric manifold $(V,\omega,\mathbb{T})$. They are, by definition, almost K\"ahler metrics such that the orthogonal distribution to the $\mathbb{T}$-orbits
is involutive (see \cite{ML}) and we will refer to such metrics as \textit{involutive almost K\"ahler metrics}. The idea of  studying such metrics comes from \cite{SKD} (see \cite{VA2} for the weighted case), where it was conjectured that the existence of a weighted involutive csc almost K\"ahler metric is equivalent to the existence of a weighted cscK metric.

\begin{iprop}[Proposition \ref{equivalence-almostcsck}]{\label{equivalence-almostcsck2}}
Let $(V, \omega,\mathbb{T})$ be a toric manifold associated to a Delzant polytope $P$. Then, for the weights defined in $(\ref{weights})$, the following statements are equivalent:

\begin{enumerate}
    \item there exists a weighted cscK  metric on $(V,\omega,\mathbb{T})$;
    \item there exists an involutive weighted csc almost K\"ahler metric on $(V,\omega,\mathbb{T})$;
    \item $P$ is weighted uniformly K-stable.
\end{enumerate}

\end{iprop}

As an application of the above result, we study the existence of extremal K\"ahler metrics on the projectivisation $\mathbb{P}(\mathcal{L}_0 \oplus \mathcal{L}_1 \oplus \mathcal{L}_2)$ of a direct sum of line bundles $\mathcal{L}_i$ over a compact complex curve $S_{\textnormal{\textbf{g}}}$ of genus $\textnormal{\textbf{g}}$. In \cite[Proposition 4]{VA3}, the authors established the existence of  involutive weighted csc almost K\"ahler metrics on $\mathbb{P}(\mathcal{L}_0 \oplus \mathcal{L}_1 \oplus \mathcal{L}_2)$, depending on the degrees of the line bundles, the genus of the basis and the K\"ahler class. Combining with Proposition \ref{equivalence-almostcsck2}, we deduce the following:

\begin{icorollary}{\label{prop-ex}}
Let $M=\mathbb{P}( \mathcal{L}_0 \oplus \mathcal{L}_1 \oplus \mathcal{L}_2) \longrightarrow S_{\textnormal{\textbf{g}}}$ be a projective $\mathbb{P}^2$-bundle over a complex curve  $S_{\textnormal{\textbf{g}}}$ of genus $\textnormal{\textbf{g}}$. If $\textnormal{\textbf{g}}=0,1$, then $M$ is a Calabi dream manifold, i.e. $M$ admits an extremal K\"ahler metric in each K\"ahler class. Furthermore, the extremal K\"ahler metrics are given by the generalized Calabi ansatz of \cite{VA2}.
\end{icorollary}

\noindent  

When $\textnormal{\textbf{g}}=0$, the existence part of Corollary $\ref{prop-ex}$ was already obtained in \cite{EL}. We prove in addition that these extremal metrics are given by the Calabi ansatz of \cite{VA2}.

\subsection{Outline of the paper} Section \ref{section-scalv} is a brief summary of the notion of weighted $(\v,\w)$-scalar curvature introduced by Lahdili \cite{AL}. In Section 3, we recall the construction and key results of semi-simple principal toric fibration established in \cite{VA2, VA3}. In Section \ref{section-distance}, we  introduce weighted distances, weighted functionals and  weighted differential operators. Section \ref{section-chencheng} gives a brief exposition of the existence result of  \cite{XC1, XC2, WH}. We  explain why their argument works equally when the properness is relative to a maximal torus of the reduced group of automorphism (and not only a connected maximal compact subgroup). In Section \ref{section-theoremA}, our main result, Theorem \ref{theoremAA}, is stated and proved. In Section 7 we review the basic facts of toric K\"ahler geometry and give the proof of Theorem  \ref{theoremBB}. In Section \ref{section-application}, we show Corollary \ref{prop-ex}.

\section*{Acknowledgement}
This paper is part of my PhD thesis. I am very grateful to my advisors Vestislav Apostolov and Eveline Legendre for their immeasurable help and invaluable advices. I would also like to thank Abdellah Lahdili for his careful reading and constructive criticism in the earlier versions, Yasufumi Nitta for his interest and to shared with me his manuscript with Shunsuke Saito, and the referee for his/her careful reading of our work and thoughtful suggestions that greatly improved the text. I am grateful to  l'Université du Québec à Montréal and l'Université Toulouse III Paul Sabatier for their financial support.

\section{The $\mathrm{v}$-scalar curvature}{\label{section-scalv}}

In this section, we review briefly the notion of \textit{weighted $\mathrm{v}$-scalar curvature} introduced by Lahdili in \cite{AL}. Consider a smooth compact K\"ahler manifold $(M,J,\omega)$. We denote by $\mathrm{Aut}_{\mathrm{red}}(M)$ the reduced group of automorphisms whose Lie algebra $\mathfrak{h}_{\red}$ is given by the ideal of real holomorphic vector fields with zeros, see \cite{PG}. Let $\mathbb{T}$ be an $\ell$-dimensional real torus in $\mathrm{Aut}_{\mathrm{red}}(M)$ with Lie algebra $\mathfrak{t}$. Suppose $\omega_0$ is a $\mathbb{T}$-invariant K\"ahler form and consider the set of smooth $\mathbb{T}$-invariant K\"ahler potentials $\mathcal{K}(M,\omega_0)^{\mathbb{T}}$ relative to $\omega_0$. For $\varphi \in \mathcal{K}(M,\omega_0)^{\mathbb{T}}$ we denote by $\omega_{\varphi}=\omega_0+dd^c\varphi$ the corresponding K\"ahler metric. It is well known that the $\mathbb{T}$-action on $M$ is $\omega_{\varphi}$-Hamiltonian (see \cite{PG}) and we let  $m_{\varphi} : M \longrightarrow \mathfrak{t}^{*}$ denote a $\omega_{\varphi}$-momentum map of $\mathbb{T}$. It is also known \cite{MA, VGSS, AL} that $P_{\varphi}:= m_{\varphi}(M)$ is a convex polytope in $\mathfrak{t}^{*}$ and we can normalize $m_{\varphi}$ by

\begin{equation}{\label{normalizing-moment-map}}
m_{\varphi}=m_{0} + d^c\varphi,
\end{equation}

\noindent in such a way that $P=P_{\varphi}$ is $\varphi$-independent, see \cite[Lemma 1]{AL}.

\begin{define}{\label{weighted-scalar}}
For $\mathrm{v}\in \mathcal{C}^{\infty}(P,\R_{>0})$ we define the (weighted) $\mathrm{v}$-scalar curvature of the K\"ahler metric $\omega_{\varphi}$, $\varphi \in \mathcal{K}(M,\omega_0)^{\mathbb{T}}$, to be

\begin{equation*}
    Scal_{\mathrm{v}}(\omega_{\varphi}):=\mathrm{v}(m_{\varphi})Scal(\omega_{\varphi})+ 2 \Delta_{\omega_{\varphi}}\big(\mathrm{v}(m_{\varphi})\big) + \textnormal{Tr}\big(G_{\varphi} \circ (\textnormal{Hess}(\mathrm{v}) \circ m_{\varphi} )\big),
\end{equation*}

\noindent where $\Delta_{\omega_{\varphi}}$ is the Riemannian Laplacian associated to $g_{\varphi}:=\omega_{\varphi}(\cdot,J\cdot)$, $\textnormal{Hess}(\mathrm{v})$ is the Hessian of $\v$ viewed as bilinear form on $\mathfrak{t}^*$  whereas  $G_{\varphi}$ is the bilinear form with smooth coefficients on $\mathfrak{t}$, given by the restriction of the Riemannian metric $g_{\varphi}$ on fundamental vector fields and $Scal(\omega_{\varphi})$ is the scalar curvature of $(M,J,\omega_{\varphi})$.
\end{define}

In a basis $\boldsymbol{\xi}=(\xi_i)_{i=1 \cdots \ell}$ of $\mathfrak{t}$ we have

\begin{equation*}
\text{Tr}\big(G_{\varphi} \circ (\text{Hess}(\v) \circ m_{\varphi} )\big) = \sum_{1\leq i,j\leq \ell}\mathrm{v}_{,ij}(m_{\varphi})g_{\varphi}(\xi_i,\xi_j)
\end{equation*}

\noindent where $\mathrm{v}_{,ij}$ stands for the partial derivatives of $\v$ in the dual basis of $\boldsymbol{\xi}$. 

\begin{define}{\label{define-scalv}}
Let $(M,J,\omega_0)$ be a compact K\"ahler manifold, $\mathbb{T}\subset \mathrm{Aut}_{\mathrm{red}}(M)$ a real torus
with normalized momentum image $P \subset \mathfrak{t}^*$ associated to $[\omega_0]$, and $\mathrm{v}\in \mathcal{C}^{\infty}(P, \R_{>0})$,
$\mathrm{w} \in \mathcal{C}^{\infty}(P, \R)$. A $(\mathrm{v}, \w)$-cscK metric is a $\mathbb{T}$-invariant K\"ahler metric satisfying

\begin{equation}{\label{weighted-cscK-metric}}
    Scal_\v(\omega_{\varphi})=\w(m_{\varphi}).
\end{equation}

\end{define}

The motivation for studying (\ref{weighted-cscK-metric}) is that many natural geometric  problems in K\"ahler geometry correspond to (\ref{weighted-cscK-metric}) for suitable choices of $\mathrm{v}$ and $\mathrm{w}$. For example, for $\mathbb{T}$ a maximal torus in $\mathrm{Aut}_{\mathrm{red}}(M)$, $\mathrm{v}\equiv1$ and $\mathrm{w}_{\textnormal{ext}}$ a certain affine-linear function on $\mathfrak{t}^*$, the $(1,\mathrm{w}_{\textnormal{ext}})$-cscK metrics are the extremal metrics in the sense of Calabi. Another example, which will be the one of the main interest of this paper, is the existence theory of extremal K\"ahler metrics on a class of toric fibrations, which can be reduced to the study of $(\mathrm{v},\mathrm{w})$-cscK on the toric fiber for suitable choices of $\mathrm{v}$ and $\mathrm{w}$. Weighted K\"ahler metrics have been extensively studied and related to a notion
of $(\v, \w)$-weighted K-stability, see for example \cite{VA7, VA6, VA5, EI, AL}.

\section{A class of toric fibrations}{\label{section-rigidtoric}}

\subsection{Semi-simple principal toric fibrations}{\label{subsection-torique-rigide}}

Let $\mathbb{T}$ be an $\ell$-dimensional torus. We denote by $\mathfrak{t}$ is Lie algebra and by $\Lambda \subset \mathfrak{t}$ the lattice of the generators of circle subgroups, so that $\mathbb{T} = \mathfrak{t}/2\pi \Lambda$. Consider $ \pi_S : Q \longrightarrow  (S,J_S)$ a principal $\mathbb{T}$-bundle over a $2d$-dimensional product of cscK Hodge manifold $(S,J_S,\omega_S)=\prod_{a=1}^k (S_a,J_a,\omega_a)$. 
 Let $\theta \in \Omega^1(Q)\otimes \mathfrak{t}$ be a connection $1$-form with curvature

\begin{equation}{\label{first-chern}}
   d \theta = \sum_{a=1}^k \pi_S^*\omega_a \otimes p_a \text{ } \text{ } \text{ } p_a \in \Lambda \subset \mathfrak{t}.
\end{equation}

\noindent The connection $1$-form $\theta$ gives rise to a \textit{horizontal} distribution $\mathcal{H}:=ann(\theta)$ and the tangent space splits as

\begin{equation*}
    TQ= \mathcal{H} \oplus \mathfrak{t},
\end{equation*}

\noindent where, by definition, $\mathcal{H}_s \overset{d_s\pi_S}{\cong} T_sS$ for all $s\in S$. The complex structure $J_S$ acts on vector fields in $\mathcal{H}$ via the unique horizontal lift from $TS$ defined via $\theta$.

Now consider a $2\ell$-dimensional compact toric K\"ahler manifold $(V,J_V,\omega,\mathbb{T})$ with associated compact Delzant polytope $P$ \cite{TD}. We will consider various actions of $\mathbb{T}$ in this paper. In order to avoid confusion, we specify on which $\mathbb{T}$ acts as a subscript, e.g. $\mathbb{T}_Q$ acts on $Q$. The interior $P^0$ is the set of regular value of the moment $m_{\omega} : V \longrightarrow P \subset \mathfrak{t}^*$ of $(V,\omega,\mathbb{T})$  and  $V^0:=m_{\omega} ^{-1}(P^0)$  is the open dense subset of points with regular $\mathbb{T}_V$-orbits. Introducing angular coordinates $t: V^0 \to \mathfrak{t} /2\pi \Lambda$  with respect to the K\"ahler structure $(J_V, \omega)$  (see e.g. \cite{MA3}), we identify 

\begin{equation}{\label{identification-moment-angle}}
V^0 \cong \mathbb{T} \times P^0 \text{  } \text{ and } \text{ } T_xV^0\cong \mathfrak{t} \oplus \mathfrak{t}^*.
\end{equation}

\noindent  for all $x \in V^0$. Notice that the first diffemorphism is $\mathbb{T}$-equivariant.

We consider the $2n=2(\ell+d)$ dimensional smooth manifold

\begin{equation*}
    M^0:=Q \times_{\mathbb{T}} V^0,
\end{equation*}

\noindent  where the $\mathbb{T}_{Q\times V^0}$-action is given by $ \gamma \cdot  (q, x) = (\gamma \cdot q, \gamma^{-1} \cdot x)$, $q \in Q$, $x \in V^0$,  and $\gamma \in  \mathbb{T}$.  Using  (\ref{identification-moment-angle}) we identify

\begin{equation}{\label{identification-M0}}
    M^0 \cong Q \times P^0.
\end{equation}

\noindent We will still denote by $\pi_S : M^0 \longrightarrow S$ the projection. At the level of the tangent space we get

\begin{equation}{\label{splitting}}
    TM^0= \mathcal{H} \oplus \mathcal{V},
\end{equation}

\noindent where, for all $s\in S$, $\mathcal{V}_s:= \textnormal{ker} d_s\pi_S \cong \mathfrak{t} \oplus \mathfrak{t}^* $ is  referred to as the \textit{vertical space}. Since $V^0$ compactifies to $V$, the smooth manifold $M^0$ compactifies to a fiber bundle with fiber $V$

\begin{equation*}
    M := \overline{M^0}= Q \times_{\mathbb{T}} V.
\end{equation*}

\noindent By construction, $M^0$ is an open dense subset of $M$ consisting of points with regular $\mathbb{T}_M$-orbits.

One can show that the almost complex structure $J_M:=J_S \oplus J_V$ on $M^0$ is integrable and extends on $M$ as $J_V$ extends to $V$.  In other words, $M$ is a compactification of a principal $(\mathbb{C}^*)^{\ell}$-bundle $\pi_S : (M^0,J) \longrightarrow (S,J_S)$.

\subsection{Compatible K\"ahler metrics}

Following \cite{VA3}, we introduce a family of K\"ahler metrics \textit{compatible with the bundle structure}.  
In momentum-angular coordinates $(m_{\omega},t)$, the K\"ahler form $\omega$ of $(V,J_V,\mathbb{T})$ is writen on $V^0$ 

\begin{equation}{\label{toric_symplectic}}
    \omega=\langle dm_{\omega} \wedge dt \rangle,
\end{equation}

\noindent where the angle bracket denotes the contraction of $\mathfrak{t}^*$ and $\mathfrak{t}$. By (\ref{identification-M0}), we can equivalently see $\theta$ on $M^0= Q \times_{\mathbb{T}} V^0$ which  satisfies $\theta(\xi^M)=\xi$ and $\theta(J\xi^M)=0$, where $\xi^M$ is the fundamental vector field defined by $\xi \in \mathfrak{t}$. Then, $\langle dm_{\omega} \wedge \theta \rangle$ is well defined on $M^0$ and restricts to $\langle dm_{\omega} \wedge dt \rangle$ on the fibers. Thus, we define more generally

\begin{equation*}
   \omega:=\langle dm_{\omega} \wedge \theta \rangle.
\end{equation*}

\noindent  We choose the real constants $c_a$, $1\leq a \leq k$, such that the affine-linear functions $\langle p_a,m_{\omega} \rangle +c_a$ are positive on $P$, where, we recall the elements $p_a \in \Lambda$ are defined by $(\ref{first-chern})$. We then define the 2-form on $M^0$

\begin{equation}{\label{metriccalabidata}}
    \tilde{\omega}=\sum_{a=1}^k\left(\langle p_a,m_{\omega}\rangle +c_a\right)\omega_a + \langle dm_{\omega} \wedge \theta \rangle,
\end{equation}

\noindent which extends to a smooth K\"ahler form on $(M,J)$ since $\omega$ does on $(V,J_V)$. In the sequel, we fix the metrics $\omega_a$, the $1$-form $\theta$ and the constants $c_a$, noting that $p_a \in \mathfrak{t}$ are topological constants of the bundle construction. The K\"ahler manifold $(M,J,\tilde{\omega},\mathbb{T})$ is then a fiber bundle over $S$, with fiber the K\"ahler toric manifold $(V, J_V, \omega, \mathbb{T})$, obtained from the principal $\mathbb{T}$-bundle $Q$. Following \cite{VA3}, we define:

\begin{define}
 The K\"ahler manifold $(M,J,\tilde{\omega},\mathbb{T})$  constructed above is referred to as \textit{a semi-simple principal toric fibration} and the K\"ahler metric given explicitly on $M^0$ by $(\ref{metriccalabidata})$, is referred to as a \textit{compatible K\"ahler metric}.  The corresponding K\"ahler classes on $(M,J)$ are called \textit{compatible K\"ahler classes} and, in the above set up, are parametrized by the real constant $c_a$.
\end{define}

\begin{remark}{\label{remark-semi-simple}}Let $(M,J, \tilde{g}, \tilde{\omega})$ be a compact K\"ahler $2n$-manifold endowed with an effective isometric hamiltonian action of an $\ell$-torus $\mathbb{T} \subset \mathrm{Aut}_{\mathrm{red}}(M)$ and momentum map $m : M \rightarrow \mathfrak{t}^{*}$. Following \cite{VA3}, we say the action is rigid if for all $x$ in $M$ $R^{*}_x\tilde{g}$ depends only on $m(x)$, where $R_x : \mathbb{T} \rightarrow \mathbb{T} \cdot x$ is the orbit map. This action is said  semi-simple rigid if moreover, for any regular value $x_0$ of the momentum map, the derivative with respect to $x$ of the family $\omega_S(x)$ of K\"ahler forms on the complex quotient $(S,J_{S})$ of $(M, J)$ is parallel and diagonalizable with respect to $\omega_{S}(x_0)$. 

By the result of \cite{VA4, VA2, VA3}, semi-simple principal toric fibrations $(M, J, \tilde{\omega}, \mathbb{T})$ correspond  to K\"ahler manifolds with a semi-simple rigid torus action, such that the K\"ahler quotient $S$ is a global product of cscK manifolds, and there are no blow-downs.
\end{remark}

\noindent The volume form of a compatible K\"ahler metric (\ref{metriccalabidata}) satisfies

\begin{equation}{\label{volume-form}}
\tilde{\omega}^{[n]}=\omega_S^{[d]}\wedge \mathrm{v}(m_{\omega}) \omega^{[\ell]}=\bigwedge_{a=1}^k \omega_a^{[d_a]}\wedge \mathrm{v}(m_{\omega}) \omega^{[\ell]}
\end{equation}

\noindent where $\mathrm{v}(m_{\omega}):=\prod_{a=1}^k\big(\langle p_a,m_{\omega} \rangle +c_a\big)^{d_a}$, $d_a$ is the complex dimension of $S_a$ and $\omega^{[i]}:=\frac{\omega^i}{i!}$ for $ 1 \leq i \leq n$. It follows from \cite{VA2} and \cite[Sect. 6]{AL}  that the scalar curvature of a compatible metric is given by

\begin{equation}{\label{weighted-scalarcurv}}
    Scal(\tilde{\omega})=\sum_{a =1}^k \frac{Scal_a}{\langle p_a,m_{\omega} \rangle +c_a} + \frac{1}{\mathrm{v}(m_{\omega})} Scal_{\v}(\omega),
\end{equation}

\noindent where $Scal_a$ is the constant scalar curvature of $(S_a,J_a,\omega_a)$ and $Scal_{\v}(\omega)$ is the $\v$-weighted scalar curvature of $(V,J_V,\omega, \mathbb{T})$, see Definition \ref{weighted-scalar}. 

\subsection{The extremal vector field}{\label{section-extremal-vector-fields}} We now recall the definition of the extremal vector field on a general compact K\"ahler manifold $M$. To this end, we fix a maximal torus $T \subset \aut_{\red}(M)$ and a K\"ahler class $[\tilde{\omega}_0]$. Given any $T$-invariant K\"ahler metric $\tilde{\omega} \in [\tilde{\omega}_0]$, we consider the $L^2_{\tilde{\omega}}$-orthogonal projection 

\begin{equation*}
   \Pi_{\tilde{\omega}} : L^2_{\tilde{\omega}} \longrightarrow P^T_{\tilde{\omega}}
\end{equation*}

\noindent where  $P^T_{\tilde{\omega}}$ is the space of $\tilde{\omega}$-Killing potentials, which, by definition, is the space of function $f\in \mathcal{C}^{\infty}(M)^T$ such that the hamiltonian vector field $X:=\tilde{\omega}^{-1}(df)$ is holomorphic. Futaki and Mabuchi \cite{FM} showed that  $\Pi_{\tilde{\omega}}\big(Scal(\tilde{\omega})\big)$ does not depend on the chosen $T$-invariant K\"ahler metric $\tilde{\omega}$ in $[\tilde{\omega}_0]$. Therefore, with respect to the normalized moment map $m_{\tilde{\omega}} : M \longrightarrow Lie(T)^*$, see (\ref{normalizing-moment-map}), one can write

\begin{equation}{\label{affin-extremal}}
     \Pi_{\tilde \omega}\big(Scal(\tilde \omega)\big)=\langle \xi_{\textnormal{ext}}, m_{\tilde{\omega}} \rangle + c_{\textnormal{ext}} =:\ell_{\textnormal{ext}}(m_{\tilde{\omega}})
\end{equation}

\noindent where $\xi_{\textnormal{ext}} \in Lie(T)$, $c_{\textnormal{ext}} \in \R$ and $\ell_{\textnormal{ext}} \in \mathrm{Aff}\big(Lie(T)^*\big)$. See \cite[Lemma 1]{AL} for more details.

Assume now $(M,J,\tilde{\omega},\mathbb{T})$ is a semi-simple principal toric fibration. Then by \cite[Proposition 1]{VA3}, the extremal vector field is tangent to the fibers, i.e. $\xi_{\mathrm{ext}} \in \mathfrak{t}$.

\begin{prop}{\label{extremal-vector-fields}}
Let $(M,J)$ be a semi-simple principal toric fibration and $T_S$ be a maximal torus in the isometry group of $g_S:=\sum_{a=1}^k g_a$, where $g_a$ is the Riemannian metric of $\omega_a$. Any compatible K\"ahler metric $\tilde{\omega}$ on $(M,J)$ is invariant by the action of a maximal torus $T \subset \aut_{\red}(M)$ such that there exists an exact sequence

 \begin{equation}{\label{exact-sequence}}
     Id \longrightarrow \mathbb{T}_M \longrightarrow T \longrightarrow T_S \rightarrow Id,
 \end{equation}
 
 \noindent where $T_S$ is a maximal torus in $\aut_{\red}(S)$. Moreover, the extremal vector field $\xi_{\textnormal{ext}}$ belongs in the Lie algebra $\mathfrak{t}$ of $\mathbb{T}_M$.
 \end{prop}

\noindent As shown in \cite{VA3}, we get from (\ref{weighted-scalarcurv}) and (\ref{affin-extremal}):

\begin{corollary}{\label{equivalence-scalar}}
A compatible K\"ahler metric $\tilde{\omega}$  on $(M,J)$ is extremal if and only if its corresponding toric K\"ahler metric $\omega$ on $(V,J_V)$ is $(\mathrm{v},\w)$-cscK   in the sense of Definition \ref{weighted-cscK-metric}, where the weights are given by

\begin{equation}{\label{weights}}
\begin{split}
	\mathrm{v}(x)  &= \prod_{a=1}^k \big(\langle p_a,x \rangle +c_a  \big)^{d_a}\\
	\w(x)&=\v(x)\left(  \ell_{\textnormal{ext}}(x) - \sum_{a=1}^k \frac{Scal_a}{\langle p_a, x \rangle +c_a} \right),
\end{split}
\end{equation}

\noindent where $\ell_{\textnormal{ext}} \in \mathrm{Aff}(P)$ is defined in $(\ref{affin-extremal})$. 
\end{corollary}

\subsection{The space of functions}

Any $\mathbb{T}_M$-invariant smooth function on $M$ pulls back to a $\mathbb{T}_Q\times \mathbb{T}_V$-invariant function on $Q\times V,$ and therefore descends to a $\mathbb{T}_V$-invariant smooth function on $S\times V$ (see $\S$\ref{subsection-torique-rigide}). This gives rise to an isomorphism of Fréchet spaces

\begin{equation}{\label{identification-functio-cartesian}}
C^{\infty}(M)^{\mathbb{T}} \cong C^{\infty}(S \times V)^{\mathbb{T}_V},
\end{equation}

\noindent which we shall tacitly use throughout the paper. Moreover, by (\ref{identification-M0}) we  get 

\begin{equation*}
    C^{\infty}(M^0)^{\mathbb{T}} \cong C^{\infty}(S \times P^0)
\end{equation*}

Given $f\in \mathcal{C}^{\infty}(M)^{\mathbb{T}}$, for any $s\in S$, we denote by $f_s \in \mathcal{C}^{\infty}(V)^{\mathbb{T}}$ the induced smooth function on $V$ with respect to the identification (\ref{identification-functio-cartesian}). Similarly,  for any $x \in V$, we denote by $f_x \in \mathcal{C}^{\infty}(S)$ the induced smooth function on $S$. It follows that on $\mathcal{C}^{\infty}(M)^{\mathbb{T}}$, the differential operator $d$ splits as $d=d_S+d_V$, where $d_S$ and $d_V$, is the exterior derivative on $S$ and $V$ respectively. We get

\begin{equation}{\label{function-from-fiber}}
\mathcal{C}^{\infty}(V)^{\mathbb{T}}\cong\{ f \in \mathcal{C}^{\infty}(M)^{\mathbb{T}} \text{ } | \text{ } d_Sf_x=0 \text{ }\forall x \in V\},
\end{equation}

\noindent showing that  $\mathcal{C}^{\infty}(V)^{\mathbb{T}}$ is closed in $ \mathcal{C}^{\infty}(M)^{\mathbb{T}}$ for the Fréchet topology.

\subsection{The space of compatible potentials}{\label{section-distance}} We fix a reference K\"ahler metric $\tilde{\omega}_0$ on $(M,J)$, its corresponding K\"ahler metric $\omega_0$ on $(V,J_V)$ and the weights $(\mathrm{v},\mathrm{w})$ given by (\ref{weights}). We denote by $\mathcal{K}(V, \omega_0)^{\mathbb{T}}$ the space of smooth K\"ahler potentials on $V$ relative to $\omega_0$ and denote by $\omega_{\varphi}=\omega_0+d_Vd_V^c\varphi$ the corresponding K\"ahler metric on $(V,J_V)$. Similarly, we denote by $\mathcal{K}(M,\tilde{\omega}_0)^{\mathbb{T}}$ the space of smooth $\mathbb{T}$-invariant K\"ahler potentials on $(M,J)$ relative to $\tilde{\omega}_0$ and we denote by $\tilde{\omega}_{\varphi}=\tilde{\omega}_0+dd^c\varphi$ the corresponding K\"ahler metric. The following Lemma is established in \cite[Lemma 7]{VA3}.

\begin{lemma}{\label{changeofmetric}}
Let $\omega_{\varphi}=\omega_0 +d_Vd^c_V \varphi$ be a $\mathbb{T}$-invariant K\"ahler metric on $(V,J_V)$ and denote by $m_{\varphi}$ the moment map which satisfies normalization (\ref{normalizing-moment-map}). Then the compatible K\"ahler metric $\tilde{\omega}_{\varphi}$ induced by $\omega_{\varphi}$ on $M$ is given by $\tilde{\omega}_{\varphi}=\tilde{\omega}_0+dd^c\varphi$, where $\varphi$ is seen as a smooth function on $M$ via $(\ref{identification-functio-cartesian})$.

\end{lemma}

 It follows that $\mathcal{K}(V, \omega_0)^{\mathbb{T}}$ parametrizes the compatible K\"ahler metric on $(M,J)$ given explicitly by $(\ref{metriccalabidata})$ and will be referred to as \textit{the space of compatible K\"ahler potentials}. It follows from Proposition \ref{extremal-vector-fields} and Lemma \ref{changeofmetric}:  

\begin{corollary}{\label{change-of-metric-cor}}
There is an embedding of Frechet spaces $ \mathcal{K}(V,\omega_0)^{\mathbb{T}} \hookrightarrow \mathcal{K}(M,\tilde{\omega}_0)^{T}$.
\end{corollary}

\section{Weighted distance, functionals and operators}{\label{section-distance}}

\subsection{Weighted distance}

Thanks to the work of Mabuchi \cite{TM1, TM} it is well-known, that $\mathcal{K}(M,\tilde{\omega}_0)^{\mathbb{T}}$ is an infinite dimensional Riemannian manifold when equipped with the Mabuchi metric: 

\begin{equation*}
  \langle \dot{\varphi}_0, \dot{\varphi}_1 \rangle_{\varphi} = \int_M \dot{\varphi}_0 \dot{\varphi}_1 \tilde{\omega}_{\varphi}^{[n]} \text{ } \text{ } \forall \dot{\varphi}_0, \dot{\varphi}_1 \in T_{\varphi} \mathcal{K}(M, \tilde{\omega}_0)^{\mathbb{T}}.
\end{equation*}

\noindent Furthermore,  a path $(\varphi_t)_{t\in[0,1]} \in  \mathcal{K}(M, \tilde{\omega}_0)^{\mathbb{T}} $ connecting two points is a smooth geodesic if and only if

\begin{equation}{\label{geodesic}}
    \Ddot{\varphi}_t - |d\dot{\varphi}_t|^2_{\varphi_t}=0.
\end{equation}

The following result is proven in \cite[Lemma 5.6]{VA6} in the more general context of semi-simple principal fiber bundles and follows easily from the expression $(\ref{volume-form})$ of the volume form of a compatible K\"ahler metric.

\begin{lemma}{\label{equality-mabuchi-norm}}
Let $\varphi \in \mathcal{K}(V, \omega_0)^{\mathbb{T}}$ and $f \in T_{\varphi}\mathcal{K}(V, \omega_0)^{\mathbb{T}}$, also viewed as an element of $T_{\varphi}\mathcal{K}(M,\tilde{\omega}_0)^{\mathbb{T}}$. Then

\begin{equation*}
 |df|^2_{\tilde{\omega}_{\varphi}}=|df|^2_{{\omega}_{\varphi}}.
 \end{equation*}

\noindent In particular, $\mathcal{K}(V,\omega_0)^{\mathbb{T}}$ is a totally geodesic submanifold of $\mathcal{K}(M,\tilde{\omega}_0)^{\mathbb{T}}$ with respect to the Mabuchi metric.

\end{lemma}

 In \cite{DG}, Guan showed the existence of a smooth geodesic between two  K\"ahler potentials on a toric manifold. The same argument shows the geodesic connectedness of  two elements $\varphi_0$, $\varphi_1 \in \mathcal{K}(V,\omega_0)^{\mathbb{T}} \subset \mathcal{K}(M,\tilde{\omega}_0)^{\mathbb{T}}$.  
 
\begin{remark}

On a general compact  K\"ahler manifold $(M, J, \omega_0)$, Darvas \cite{DT} introduced the distance $d_1$ as

\begin{equation}{\label{distance-d1}}
d_1(\varphi_0,\varphi_1):=\inf_{\varphi_t} \int_0^1 \int_M | \dot{\varphi}_t| \omega^{[n]}_t,
\end{equation}    

\noindent where $\omega_t^{[n]}$ is the volume form associated to the metric $\omega_t=\omega+dd^c\varphi_t$ and the infimum is taken over the space of smooth curves $\{\varphi_t\}_{t\in[0,1]} \subset  \mathcal{K}(M,\omega_0)^{\mathbb{T}}$ joining $\varphi_0$ to $\varphi_1$. In the above formula, $\dot{\varphi}_t$ is the variation of $\varphi_t$ with respect to $t$. It is showed in \cite{TD} that  $d_1(\varphi_0, \varphi_1)$ equals to the length of the unique (weak) $C^{1, \bar 1}$ geodesic \cite{XC} joining $\varphi_0$ and $\varphi_1$.
\end{remark}

\begin{lemma}{\label{restriction-distance}}
The distance $d_1$ restricts (up to a positive multiplicative constant) to the distance $d_{1,\v}$ on $\mathcal{K}(V, \omega_0)^{\mathbb{T}}$, defined by

\begin{equation}{\label{distance-pondéré}}
   d^V_{1,\mathrm{v}}(\varphi_0,\varphi_1):=\inf_{\varphi_t} \int_0^1 \int_V | \dot{\varphi}_t| \mathrm{v}(m_t)\omega_t^{[\ell]}.
\end{equation}

\end{lemma}

\noindent We refer to \cite[Corollary 5.5]{VA6} for the proof. It follows also directly from $(\ref{volume-form})$ and the smooth geodesic connectedness of $\mathcal{K}(V, \omega_0)^{\mathbb{T}}$.

\subsection{Weighted functionals}{\label{section-functionals}}

\noindent We consider the Mabuchi energy on $\mathcal{K}(M,\tilde{\omega}_0)^{\mathbb{T}}$ relative to $\mathbb{T}$, characterized by its variation

\begin{equation}{\label{define-mabuchi}}
	d_{\varphi}\mathcal{M}^{\mathbb{T}}(\Dot{\varphi}) = - \int_M \Dot{\varphi} \big(Scal(\tilde{\omega}_{\varphi}) -  \Pi_{\tilde{\omega}_{\varphi}}\big(Scal(\tilde{\omega}_{\varphi})\big) \tilde{\omega}_{\varphi}^{[n]}, \text{ } \text{ }
		\mathcal{M}^{\mathbb{T}}(0)=0.
\end{equation}

\noindent When restricted to $\mathcal{K}(V, \omega_0)^{\mathbb{T}} \subset \mathcal{K}(M, \tilde \omega_0)^{\mathbb{T}}$, this functional is a special case of the weighted Mabuchi functional introduced in \cite{AL}. In our case, the following lemma, established in \cite{VA3}, follows directly from (\ref{weighted-scalarcurv}) and (\ref{weights}).

\begin{lemma}{\label{mabuchi-energy-restriction}}
 The restriction of the Mabuchi relative energy $\mathcal{M}^{\mathbb{T}}$ to $\mathcal{K}(V, \omega_0)^{\mathbb{T}}$ is equal (up to a  positive multiplicative constant) to the weighted Mabuchi energy, defined by

\begin{equation}{\label{definition-weighted-eneergy}}
	d_{\varphi}\mathcal{M}_{\mathrm{v},\mathrm{w} }(\Dot{\varphi}) := - \int_V\big(Scal_{\mathrm{v}}(\omega_{\varphi}) - \w(m_{\varphi})\big)\Dot{\varphi}\omega_{\varphi}^{[\ell]}, \text{ } \text{ } \mathcal{M}_{\mathrm{v},\mathrm{w} }(0)=0,
\end{equation}

\noindent with weights $(\v,\w)$ given by $(\ref{weights})$,  and where $\varphi \in \mathcal{K}(V, \omega_0)^{\mathbb{T}}$ and $\dot{\varphi} \in T_{\varphi}\mathcal{K}(V, \omega_0)^{\mathbb{T}}$. In particular,   compatible extremal K\"ahler metrics in $[\tilde{\omega}_0]$ are critical points of $\mathcal{M}_{\mathrm{v},\mathrm{w}} : \mathcal{K}(V,\omega)^{\mathbb{T}} \longrightarrow \R$.
\end{lemma}

 The Aubin-Mabuchi functional $\mathcal{I} : \mathcal{K}(M,\tilde{\omega}_0)^{\mathbb{T}}  \longrightarrow \R$ is defined by

\begin{equation*}\label{I-functionnal}
d_{\varphi}\mathcal{I}(\Dot{\varphi}) = \int_M\Dot{\varphi}\tilde{\omega}_{\varphi}^{[n]}, \text{ } \text{ }  \mathcal{I}(0)=0, 
\end{equation*}

\noindent for any $\dot{\varphi} \in T_{\varphi}\mathcal{K}(M,\tilde{\omega}_0)^{\mathbb{T}}$. By (\ref{volume-form}), its restriction to $\mathcal{K}(V,\omega_0)^{\mathbb{T}} \subset \mathcal{K}(M,\tilde{\omega}_0)^{\mathbb{T}}$ is equal to (up to a  positive multiplicative constant)

\begin{equation}{\label{Ir-functionnal}}
d_{\varphi}\mathcal{I}_{\v}(\Dot{\varphi}) := \int_V\Dot{\varphi}\v(m_{\varphi})\omega_{\varphi}^{[\ell]}, \text{ } \mathcal{I}_{\v}(0)=0
\end{equation}

\noindent for any $\dot{\varphi} \in T_{\varphi}\mathcal{K}(V, \tilde{\omega}_0)^{\mathbb{T}}$. We define the space of $\mathcal{I}$-normalized relative K\"ahler potentials as

\begin{equation}{\label{normalized-potential}}
    \mathring{\mathcal{K}}(M,\tilde{\omega}_0)^{\mathbb{T}}:=\mathcal{I}^{-1}(0) \subset \mathcal{K}(M,\tilde{\omega}_0)^{\mathbb{T}}.
\end{equation}

\noindent It is well known, see e.g. \cite[chapter 4]{PG}, that this space is totally geodesic in $\mathcal{K}(M,\tilde{\omega}_0)^{\mathbb{T}}$. Similarly, we define

\begin{equation}{\label{normalized-compatible-potential}}
  \mathring{\mathcal{K}}_{\v}(V,\omega_0)^{\mathbb{T}}:=\mathcal{I}_{\v}^{-1}(0) \subset \mathcal{K}_{\v}(V,\omega_0)^{\mathbb{T}}.  
\end{equation}

\noindent It follows from $(\ref{Ir-functionnal})$ that we also have $\mathring{\mathcal{K}}_{\v}(V,\omega_0)^{\mathbb{T}} \subset \mathring{\mathcal{K}}(M,\tilde{\omega}_0)^{\mathbb{T}} $.

\subsection{Weighted differential operators}{\label{section-operator}}

Following \cite{VA3}, we introduce the $\v$-Laplacian of $(V,J_V,\omega)$ acting on smooth function

\begin{equation*}
    \Delta^V_{\omega,\v}f:=\frac{1}{\v(m_{\omega})}\delta \big(\v(m_{\omega})d_Vf\big),
\end{equation*}

\noindent where $\delta$ is the formal adjoint of the differential $d_V$ with respect to $\omega^{[\ell]}$. This definition immediately implies that $\Delta^V_{\omega,\v}$ is self-adjoint with respect to $\v(m_{\omega})\omega^{[\ell]}$. Moreover, it follows from the computations in \cite[Lemma~8]{VA3} that $\Delta^V_{\omega,\v}$ can be alternatively expressed as

\begin{equation}{\label{expression-v-laplacien}}
    \Delta^V_{\omega,\v}f= \Delta_{\omega}f - \sum_{a=1}^k \frac{d_ad_V^cf(p^V_a)}{\langle p_a,m_{\omega} \rangle + c_a}
\end{equation}

\noindent \noindent for any $f \in \mathcal{C}^{\infty}(V)$, where $\Delta_{\omega}$ is the Laplacian with respect to $\omega$ and $p_a^V$ is the fundamental vector field on $V$ defined by $p_a \in \mathfrak{t}$. As in \cite{AL}, we introduce the $\v$-weighted Lichnerowicz operator of $(V,J_V,\omega)$ defined on the smooth functions $f \in \mathcal{C}^{\infty}(V)$ to be

\begin{equation}{\label{define-lichne-pondéré}}
    \mathbb{L}^V_{\omega,\v}f:=\frac{\delta \delta\big(\v(m_{\omega})(D^-d_Vf)\big)}{\v(m_{\omega})},
\end{equation}

\noindent where $D$ is the Levi-Civita connection of $\omega$, $D^-d_V$ denotes the $(2,0)+(0,2)$ part of $Dd_V$ and $\delta : \otimes^pT^*V \longrightarrow \otimes^{p-1}T^*V$ is defined in any local orthogonal frame $\{e_1,\dots, e_{2n} \}$  by

\begin{equation*}
    \delta \psi  := - \sum_{i=1}^{2n} e_i \lrcorner D_{e_i} \psi
\end{equation*}

\noindent where $\lrcorner$ denotes the interior product. The operator $\delta \delta$ is the formal adjoint of $D^-d_V$ with respect to $\omega^{[\ell]}$.  Hence, the $\v$-weighted  Lichnerowicz operator is  self-adjoint with respect to the volume form $\v(m_{\omega})\omega^{[\ell]}$. Let $\tilde{\omega}$ be the compatible K\"ahler metric on $(M,J)$ corresponding to $\omega$. We denote by $\omega_S(x)$ the K\"ahler form on $(S,J_S)$ induced by $\tilde{\omega}$:

\begin{equation*}
    \omega_S(x):= \sum_{a =1}^k\big(\langle p_a, m_{\omega}(x) \rangle +c_a \big) \omega_a.
\end{equation*}

\noindent  The following is etablished in the proof of \cite[Lemma 8]{VA3}.

\begin{prop}{\label{decomposition-Lichne-prop}}
Let $f$  be a $\mathbb{T}$-invariant smooth function on $M$, seen as a $ \mathbb{T}_V$-invariant function on $V\times S$ via $(\ref{identification-functio-cartesian})$. We denote by $\Delta_{\tilde{\omega}}$ the Laplacian of $(M,J,\tilde{\omega})$ and by $\Delta_x^S$, respectively $\mathbb{L}^S_x$, the Laplacian, respectively the Lichnerowicz operator, of $(S,J_S,\omega_S(x))$. We then have

\begin{equation*}
    \Delta_{\tilde{\omega}}f=\Delta^V_{\omega,\v}f_s + \Delta_x^Sf_x
\end{equation*}

\noindent Furthermore, the corresponding Licherowicz's operators  $\mathbb{L}_{\tilde{\omega}}$, $\mathbb{L}^V_{\omega,\v}$ and $\mathbb{L}_x^S$ are related by

\begin{equation}{\label{decomposition-Lichne}}
\mathbb{L}_{\tilde{\omega}}f=\mathbb{L}^V_{\omega,\v}f_s+\mathbb{L}^S_xf_x+\Delta^S_x\big(\Delta^V_{\omega,\v}f_s\big)_x + \Delta^V_{\omega}\big(\Delta_x^S f_x\big)_s  + \sum_{a=1}^kQ_a(x)\Delta_{a}f_x
\end{equation}

\noindent where $\Delta _a$ is the Laplacian with respect to $(S_a,J_a,\omega_a)$ and $Q_a(x)$ is a smooth function on $V$.
\end{prop}

\noindent We fix a compatible K\"ahler metric 

\begin{equation*}
\tilde{\chi}=\sum_{a=1}^k (\langle p_a, m_{\chi} \rangle  + c_{a,\alpha})\omega_a + \chi
\end{equation*}

\noindent corresponding to a K\"ahler metric $\chi$ on $(V,J_V)$, where $m_{\chi}$ is a moment map with respect to $\chi$ and $c_{a, \alpha}$ are  constants depending on $\alpha:=[\tilde{\chi}]$ such that $\langle p_a, m_{\chi} \rangle  + c_{a,\alpha}>0$.
Hashimoto introduced in \cite{YH} the operator $\mathbb{H}_{\tilde{\omega}}^{\tilde{\chi}} : \mathcal{C}^{\infty}(M)^{\mathbb{T}} \longrightarrow \mathcal{C}^{\infty}(M)^{\mathbb{T}}$ defined by

\begin{equation}{\label{Hashimoto-operator-definition}}
  \mathbb{H}^{\tilde{\chi}}_{\tilde{\omega}}f:=  g_{\tilde{\omega}}\big(\tilde{\chi}, dd^c f\big) + g_{\tilde{\omega}}\big(d \Lambda_{\tilde{\omega}} \tilde{ \chi}, df\big).
\end{equation}

\noindent According to \cite[Lemma~1]{YH},  $\mathbb{H}^{\tilde{\chi}}_{\tilde{\omega}}$ is a second order elliptic self-adjoint  differential operator with respect to $\tilde{\omega}^{[n]}$. Furthermore, the kernel of $\mathbb{H}^{\tilde{\chi}}_{\tilde{\omega}}$ is the space of constant functions. We define the  $\v$-weighted \textit{Hashimoto operator} $\mathbb{H}^{\chi}_{\omega,\v} : \mathcal{C}^{\infty}(V)^{\mathbb{T}} \rightarrow \mathcal{C}^{\infty}(V)^{\mathbb{T}}$ by

\begin{equation*}
  \mathbb{H}^{\chi}_{\omega,\v}f:= g_{\omega}\big(\chi, d_Vd_V^c f \big) + g_{\omega}\big(d_V \Lambda_{\omega} \chi, d_Vf\big)  + \frac{1}{\v(m_{\omega})}g_{\omega}\big(\chi,d_V\v(m_{\omega}) \wedge d_V^cf\big).
\end{equation*}

\begin{prop}{\label{decomposition-Hashimoto}}

Let $f$  be a $\mathbb{T}$-invariant smooth function on $M$, seen as a $ \mathbb{T}_V$-invariant function on $V\times S$ via $(\ref{identification-functio-cartesian})$. The Hashimoto operator admits the following decomposition 

\begin{equation*}
    \mathbb{H}^{\tilde{\chi}}_{\tilde{\omega}}f= \mathbb{H}_{\omega,\v}^{\chi}f_s +  \sum_{a=1}^k R_a(x)\Delta_a f_x,
\end{equation*}

\noindent where $R_a(x)$ is a smooth function on $V$ depending on $\chi$ and $\alpha$.
\end{prop}

\begin{proof}

\noindent For simplicity, we denote by $m$ the moment map of $\omega$ and

\begin{equation*}
    q(m):=\sum_{a=1}^k \frac{d_a \big( \langle p_a,m_{\chi} \rangle +c_{a,\alpha}\big) }{\langle p_a,m \rangle +c_{a}}.
\end{equation*}

Let $K \in \mathcal{C}^{\infty}(V)^{\mathbb{T}} \otimes \mathfrak{t}^*$ the generator of the $\mathbb{T}_V$-action.  By definition $d_V^cf(K)$ is a smooth $\mathbb{T}_V$-invariant $\mathfrak{t}^*$-valued function on $V$ and induces a smooth $\mathbb{T}_M$-invariant $\mathfrak{t}^*$-valued function on $M$ via (\ref{identification-functio-cartesian}). It is shown in the proof of \cite[Lemma 8]{VA3} that on $M^0$

\begin{equation}{\label{decomposition-ddc}}
\begin{split}
        dd^cf=& \langle d_V(d^c_Vf_s(K))_s \wedge \theta \rangle + \langle d_S(d^c_Vf_s(K))_x \wedge \theta \rangle \\
        &+ \sum_{a=1}^k  d^c_Vf_s(p_a^V)  \omega_a + d_Sd^c_Sf_x + \langle d^c_S(d^c_Vf_s(K))_x,J\theta \rangle.
\end{split}
\end{equation}

\noindent First, we recall the general identity

\begin{equation}{\label{1er-term-produit}}
    g_{\tilde{\omega}}(dd^cf,\tilde{\chi})\tilde{\omega}^{[n]}=-dd^cf \wedge \tilde{\chi} \wedge \tilde{\omega}^{[n-2]} - \Delta_{\tilde{\omega}}f \Lambda_{\tilde{\omega}}(\tilde{\chi}) \tilde{\omega}^{[n]}.
\end{equation}

\noindent  From the expression of $\tilde{\chi}$ and $\tilde{\omega}$, we can see that 

\begin{equation*}
  \bigg( \langle d_S(d^c_Vf_s(K))_x \wedge \theta \rangle +  \langle d^c_S(d^c_Vf_s(K))_x,J\theta \rangle\bigg) \wedge \tilde{\chi} \wedge \tilde{\omega}^{[n-2]}=0.
\end{equation*}

\noindent A straightforward computation gives
 
  \begin{equation}{\label{trace-calcul-equation-nvx}}
     \Lambda_{\tilde{\omega}}(\tilde{\chi})=\Lambda_{\omega}(\chi) + \sum_{a=1}^k \frac{d_a \big( \langle p_a,m_{\chi} \rangle +c_{a,\alpha}\big) }{\langle p_a,m \rangle +c_{a}}.
 \end{equation}

\noindent From Proposition \ref{decomposition-Lichne-prop}, (\ref{trace-calcul-equation-nvx}) and (\ref{1er-term-produit}) we have

\begin{equation}
    g_{\tilde{\omega}}(dd^cf_s,\tilde{\chi}) = g_{\omega}(d_Vd^c_Vf_s,\chi) + \sum_{a=1}^k\frac{d_a d_Vf_s(p_a^V) (\langle p_a,m_{\chi}\rangle +c_{a,\alpha})}{(\langle p_a,m\rangle +c_{a})^2}.
\end{equation}

\noindent  Using (\ref{trace-calcul-equation-nvx}) we get

\begin{equation}{\label{enfin-hasimoto}}
    g_{\tilde{\omega}}\big(d\Lambda_{\tilde{\omega}}(\tilde{\chi}),df_s\big)=   g_{\omega}\big(d_V\Lambda_{\omega}(\chi),d_Vf_s\big)+  g_{\omega}\big(d_Vq(m), d_V^cf_s \big).
\end{equation}

\noindent To summarize, we have shown

\begin{equation}{\label{Hashimoto-1}}
\begin{split}
      \mathbb{H}^{\tilde{\chi}}_{\tilde{\omega}}f_s=& g_{\omega}(d_Vd^c_Vf_s,\chi)+ g_{\omega}\big(d_V\Lambda_{\omega}(\chi),d_Vf_s\big)\\
      &+  g_{\omega}\big(d_Vq(m) , d_V^cf_s \big)  + \sum_{a =1}^k\frac{d_a d_Vf_s(p_a^V) (\langle p_a,m_{\chi}\rangle +c_{a,\alpha})}{(\langle p_a,m\rangle +c_{a})^2}.
\end{split}      
\end{equation}

\noindent Using (\ref{weights})  we have

\begin{equation}{\label{terme-en-trop}}
   \frac{1}{\v(m)} g_{\omega}\big(\chi,d_V\v(m) \wedge d_V^cf_s\big)=  g_{\omega}\big(d_Vq(m) , d_V^cf_s \big)  + \sum_{a =1}^k\frac{d_a d_Vf_s(p_a^V) (\langle p_a,m_{\chi}\rangle +c_{a,\alpha})}{(\langle p_a,m\rangle +c_{a})^2}. 
\end{equation}

\noindent From $(\ref{Hashimoto-1})$ and $(\ref{terme-en-trop})$ we get 

\begin{equation*}
  \mathbb{H}^{\tilde{\chi}}_{\tilde{\omega}}f_s  = \mathbb{H}^{\chi}_{\omega,\v}f_s.
\end{equation*}

\noindent The term $\mathbb{H}^{\tilde{\chi}}_{\tilde{\omega}}f_x$ is obtained via similar computation.

\end{proof}

\begin{remark}{\label{remark-Hashimoto}}
Proposition \ref{decomposition-Lichne-prop} implies in particular that the restriction of $\mathbb{H}^{\tilde{\chi}}_{\tilde{\omega}}$ to the Frechet subspace $\mathcal{C}^{\infty}(V)^{\mathbb{T}} \subset \mathcal{C}^{\infty}(M)^{\mathbb{T}}$ coincides with $\mathbb{H}_{\omega,\v}^{\chi}$. It follows that $\mathbb{H}_{\omega,\v}^{\chi}$ is a self-adjoint (with respect to $\v(m_{\omega})\omega^{[\ell]}$) second order elliptic operator.
\end{remark}

\section{An analytic criterion for the existence of extremal K\"ahler metrics}{\label{section-chencheng}}

In this section we recall the existence  results of extremal K\"ahler metrics in a given K\"ahler class, proved by Chen--Cheng \cite{XC1, XC2} in the constant scalar curvature case and extended by He \cite{WH} to the extremal case. 

We fix a compact complex manifold $(M,J)$, a maximal compact connected subgroup $K$ of $ \aut_{\red}(M)$ and a $K$-invariant K\"ahler metric $\omega_0$. Let $\xi_{\textnormal{ext}}$ denotes the corresponding extremal vector field, as explained in \ref{section-extremal-vector-fields}. Since the extremal vector field $\xi_{\textnormal{ext}}$ is central in the Lie algebra of $K$, it generates a group $\mathbb{T}_{\textnormal{ext}}$ in the center of the complexified group $G:=K^{\mathbb{C}}$ of
$K$. As in  \cite{WH}, we consider the space of $\mathbb{T}_{\textnormal{ext}}$-invariant K\"ahler potentials $\mathcal{K}(M,\omega_0)^{\mathbb{T}_{\textnormal{ext}}}$. The
groupe $G$ acts on $\mathring{\mathcal{K}}(M,\omega_0)^{\mathbb{T}_{\textnormal{ext}}}$  via the natural action on K\"ahler metrics in $[\omega_0]$ and the normalization  (\ref{normalized-potential}). We
introduce the distance $d_{1,G}$ relative to $G$

\begin{equation}{\label{distance-relative}}
d_{1,G}(\varphi_1,\varphi_2):=\inf_{\gamma \in G}d_1(\varphi_1, \gamma \cdot \varphi_2),
\end{equation}
  
\noindent where $d_1$ is defined in (\ref{distance-d1}). Let $\mathcal{M}^{\mathbb{T}_{\textnormal{ext}}}$ be the Mabuchi energy relative to $\mathbb{T}_{\textnormal{ext}}$, see $(\ref{define-mabuchi})$. We recall the following definition from \cite{DR}:

\begin{define}{\label{def-proper}} The relative Mabuchi energy $\mathcal{M}^{\mathbb{T}_{\textnormal{ext}}}$ is  said proper with respect to $d_{1,G}$ if
\begin{itemize}
    \item $\mathcal{M}^{\mathbb{T}_{\textnormal{ext}}}$ is bounded from below on $\mathcal{K}(M, \omega_0)^{\mathbb{T}_{\textnormal{ext}}}$;
    \item for any sequence $\varphi_i \in \mathring{\mathcal{K}}(M,\omega_0)^{\mathbb{T}_{\textnormal{ext}}}$, $d_{1,G}(0, \varphi_i) \rightarrow \infty$ implies that $\mathcal{M}^{\mathbb{T}_{\textnormal{ext}}}(\varphi_i)\rightarrow \infty$.
\end{itemize}

\end{define}

\begin{theorem}{\label{Chen--Cheng-existence}}
The relative Mabuchi energy $\mathcal{M}^{\mathbb{T}_{\textnormal{ext}}}$ restricted to $\mathcal{K}(M,\omega_0)^K \subset  \mathcal{K}(M,\omega_0)^{{\mathbb{T}_{\textnormal{ext}}}}$ is  $d_{1,G}$-proper if and only if there exists an extremal K\"ahler metric in $(M,J,[\omega_0])$ with extremal vector fields $\xi_{\textnormal{ext}}$.  Moreover, the same assertion holds by replacing $\mathbb{T}_{\textnormal{ext}}$ by a maximal torus $T \subset \aut_{\red}(M)$ in the maximal compact group $K$ and $G=K^{\mathbb{C}}$ by the complexification $T^{\mathbb{C}}$ of $T$.
\end{theorem}

The first assertion is established in \cite[Theorem 3.1]{WH}. We can directly modify the argument to obtain the second. Indeed,  in the one direction, suppose that $\mathcal{M}^T$ is $T^{\mathbb{C}}$-proper, in the sense that $\mathcal{M}^T$ is bounded from below on $\mathcal{K}(M,\omega_0)^T$ and for any sequence $\varphi_i \in \mathring{\mathcal{K}}(M,\omega_0)^{T}$, $d_{1,T^{\mathbb{C}}}(0, \varphi_i) \rightarrow \infty$ implies that $\mathcal{M}^{T}(\varphi_i)\rightarrow \infty$. Since $T \subset K$ is a maximal torus  it must contain the
center of $K$, i.e. $\mathbb{T}_{\textnormal{ext}} \subset T$.  Hence,  $\mathcal{M}^T|_{\mathcal{K}(M,\omega_0)^{K}}=\mathcal{M}^{\mathbb{T}_{\textnormal{ext}}}|_{\mathcal{K}(M,\omega_0)^{K}}$, where $\mathcal{K}(M,\omega_0)^{K} \subset \mathcal{K}(M, \omega_0)^T$ is the subspace of $K$-invariant $\omega_0$-relative K\"ahler potentials.
As any $T^{\mathbb{C}}$-orbit  of an element of $\mathcal{K}(M,\omega_0)^{K}$ belongs to its $G$-orbit,  the $d_{1,T^{\mathbb{C}}}$-properness of $\mathcal{M}^T$  implies that $\mathcal{M}^{\mathbb{T}_{\textnormal{ext}}}$ is 
$d_{1,G}$-proper when restricted to the subspace $\mathcal{K}(M,\omega_0)^{K}$.  By \cite[Theorem 3.1]{WH}, this implies  the existence of a $K$-invariant (and hence $T$-invariant) extremal K\"ahler metric in $[\omega_0]$.

Conversely, suppose $[\omega_0]$ admits a $T$-invariant extremal K\"ahler metric. Then the proof of  \cite[Theorem~3.7]{WH} yields the $T^{\mathbb{C}}$-properness of $\mathcal{M}^T$, should one have the uniqueness of the $T$-invariant extremal K\"ahler metrics modulo $T^{\mathbb{C}}$. Generalizing the result of Berman-Berndtsson \cite{BB} and Chen-Paun-Zeng \cite{CPZ}, Lahdili showed, in the more general context of $(\v,\w)$-weighted metrics \cite[Theorem~2, Remark~2]{AL2}, that the $T$-invariant extremal metrics are unique modulo the action of $T^{\mathbb{C}}$.

\section{An analytic criterion in the case of semi-simple principal toric fibrations}{\label{section-theoremA}}

\noindent This section is devoted to prove of the following result (where we use notation of \S\ref{section-rigidtoric}).

\begin{theorem}{\label{theoremA}}

Let $(M,J, \tilde{\omega}_0, \mathbb{T})$ be a semi-simple principal toric fibration with K\"ahler toric fiber $(V,J_V, \omega_0, \mathbb{T} )$ and let $(\v, \w)$ be the corresponding weight functions defined in $(\ref{weights})$. Then, the following statements are equivalent:

\begin{enumerate}
    \item there exists an extremal K\"ahler metric in $(M,J,[\tilde{\omega}_0],\mathbb{T})$;
      \item  there exists a compatible extremal K\"ahler metric in $(M,J,[\tilde{\omega}_0],\mathbb{T})$;
    \item  there exists a $(\v,\w)$-cscK metric in $(V,J_V,[\omega_0],\mathbb{T})$.

\end{enumerate}

\end{theorem}

The statement $(2) \Leftrightarrow (3)$ is established in Corollary \ref{equivalence-scalar} whereas the statement $(2)\Rightarrow(1)$ is clear.  We focus on $(1) \Rightarrow (2)$.

We follow the argument of He \cite{WH} by restricting the continuity path of Chen \cite{XC5} to compatible K\"ahler metrics. We consider the continuity path for $ \varphi \in \mathcal{K}(V,\omega_0)^{\mathbb{T}} \subset \mathcal{K}(M,\tilde{\omega}_0)^{\mathbb{T}}$ given by

\begin{equation}{\label{continuity-path-weithed}}
  t\big(Scal_{\v}(\omega_{\varphi})-\w(m_{\varphi})\big) = (1-t)\big(\Lambda_{\omega_{\varphi},\v}(\chi)-n\big), \text{ } \text{ } t\in[0,1],
  \end{equation}

\noindent for some K\"ahler metric $\chi \in [\omega_0]$ that we will wisely choose in $(\ref{choice-chi})$. In the above formula

\begin{equation*}
 \Lambda_{\omega_{\varphi},\v}(\chi):=\Lambda_{\omega_{\varphi}}(\chi) +  \sum_{a=1}^k \frac{d_a \big( \langle p_a,m_{\chi} \rangle +c_{a}\big) }{\langle p_a,m_{\varphi} \rangle +c_{a}}
\end{equation*}

\noindent is a smooth function on $V$ equal to $\Lambda_{\tilde{\omega}_{\varphi}}(\tilde{\chi})$. By definition, a solution $\varphi_t$ at $t=1$ corresponds to a compatible extremal metric on $(M,J)$ or equivalently to a $(\v,\w)$-cscK on $(V,J_V)$. For $t_1 \in (0,1]$, we define

\begin{equation}{\label{set-solution}}
    S_{t_1}:=\{ t \in (0,t_1] \text{ }| \text{ } (\ref{continuity-path-weithed}) \text{ has a solution } \varphi_t \in \mathcal{K}(V,\omega_0)^{\mathbb{T}}\}.
\end{equation}

\noindent We need to show that $S_1$ is open, closed and non empty.

\subsection{Openness}

\begin{prop}{\label{openess}}
$S_1$ is open and non empty.
\end{prop}

\noindent For a compatible K\"ahler form $\tilde{\omega}$ on $(M,J)$ corresponding to a K\"ahler metric $\omega$ on $(V,J_V)$, we denote by  $C^{\infty}(M,\tilde{\omega})^{\mathbb{T}}$ the space of $\mathbb{T}_M$-invariant smooth functions with zero mean value with respect to $\tilde{\omega}^{[n]}$ and by  $\mathcal{C}_{\v}^{\infty}(V, \omega)^{\mathbb{T}} \subset C^{\infty}(M,\tilde{\omega})^{\mathbb{T}}$ the space of $\mathbb{T}_V$-invariant smooth functions with zero mean value with respect to $\v(m_{\omega})\omega^{[\ell]}$. The following is an adaptation of  \cite[Lemma~3.2]{WH}.

\begin{lemma}{\label{prop-starting-point}}
$S_1$ is non empty.
\end{lemma}

\begin{proof}

Let $\omega$ a K\"ahler metric on $(V,J_V)$ and $\tilde{\omega}$ its associate compatible K\"ahler metric on $(M,J)$ via (\ref{metriccalabidata}). Since $\Delta_{\omega,\v}^V$ is self-adjoint with respect to $\v(m_{\omega})\omega^{[\ell]}$, it follows from the proof of Proposition \ref{operator-iso} below that

\begin{equation}{\label{iso-laplacian-poid}}
  \Delta_{\omega,\v}^V : \mathcal{C}_{\v}^{\infty}(V, \omega)^{\mathbb{T}} \longrightarrow \mathcal{C}_{\v}^{\infty}(V, \omega)^{\mathbb{T}}  
\end{equation}

\noindent is an isomorphism. Denote by $f \in C^{\infty}(M,\tilde{\omega})^{\mathbb{T}}$ the unique solution of

\begin{equation}{\label{poisson-solution}}
\Delta_{\tilde{\omega}} f = Scal_{\v}(\omega) - \w(m_{\omega}).
\end{equation}

\noindent  By $(\ref{iso-laplacian-poid})$,  $f \in \mathcal{C}_{\v}^{\infty}(V, \omega)^{\mathbb{T}}$. Now we choose

\begin{equation}{\label{choice-chi}}
   \tilde{ \chi}:=\tilde{\omega}-dd^c\frac{f}{r}.
\end{equation}

 Since $f$ is a $\mathbb{T}_V$-invariant smooth function on $V$, $\tilde{\chi}$ is both K\"ahler and compatible  for $r$ sufficiently large by Lemma \ref{changeofmetric}. We denote by $\chi$ its corresponding K\"ahler metric on $(V,J_V)$. Then

\begin{equation*}
\begin{split}
\Delta_{\tilde{\omega}}f=&r\Delta_{\tilde{\omega}}f\frac{1}{r} = -r\Lambda_{\tilde{\omega}}dd^c\frac{f}{r} \\
=& r \Lambda_{\tilde{\omega}} \bigg(\tilde{\omega}-dd^c\frac{f}{r} - \tilde{\omega} \bigg) \\
=&  r\big( \Lambda_{\omega,\v}(\chi) - n \big). 
\end{split}
\end{equation*}

\noindent Now let us write $r=t_0^{-1}-1$, for $t_0 \in (0,1)$ sufficiently small.  Then $(\omega,t_0)$ is solution of $(\ref{continuity-path-weithed})$.

\end{proof}

Now we show that $S_1$ is open. We fix $(\omega_{t_0},t_0)$ the solution of (\ref{continuity-path-weithed}) given by Lemma~\ref{prop-starting-point}. Let  $\tilde{\omega}_{t_0}=\omega_0 +dd^c\varphi_{t_0}$ be its associated compatible K\"ahler metric on $(M,J)$, with $\varphi_{t_0} \in \mathcal{K}(V,\omega)^{\mathbb{T}}$.  Let $\pi : \mathcal{C}^{\infty}(M)^{\mathbb{T}} \longrightarrow \mathcal{C}^{\infty}(M,\tilde{\omega}_{t_0})^{\mathbb{T}}$  be the linear projection:

\begin{equation*}
    \pi(f):= f - \frac{1}{\int_M  \tilde{\omega}^{[n]}_{t_0}} \int_M f \tilde{\omega}^{[n]}_{t_0}.
\end{equation*}

\noindent We consider

\begin{equation*}
R : \mathring{\mathcal{K}}(M,\tilde{\omega}_0)^{\mathbb{T}} \times [0,1] \longrightarrow  \mathcal{C}^{\infty}(M)^{\mathbb{T}},
\end{equation*}

\noindent defined by

\begin{equation*}
    R(\varphi,t):= t\big(Scal(\tilde{\omega}_{\varphi})  -  \Pi_{\tilde{\omega}_{\varphi}}(Scal(\tilde{\omega}_{\varphi})\big) - (1-t)(\Lambda_{\tilde{\omega}_{\varphi}}(\tilde{\chi})-n).
\end{equation*}

 \noindent  The linearization of the composition $\pi \circ R$ at $(\varphi_{t_0},t_0)$ is given by 

\begin{equation}
D (\pi \circ R)(\varphi_{t_0},t_0)[f,s]= \pi \bigg(\mathcal{L}_{\tilde{\omega}_{t_0}}f+s\bigg(Scal(\tilde{\omega}_{t_0}) -  \Pi_{\tilde{\omega}_{t_0}}\big(Scal(\tilde{\omega}_{t_0})\big) +\Lambda_{\tilde{\omega}_{t_0}}(\tilde{\chi})-n \bigg)\bigg),
\end{equation}

\noindent where

\begin{equation*}
\mathcal{L}_{\tilde{\omega}_{t_0}}=-2t_0\mathbb{L}_{\tilde{\omega}_{t_0}}+(1-t_0)\mathbb{H}^{\tilde{\chi}}_{\tilde{\omega}_{t_0}}.
\end{equation*}

 \noindent Above we used the notation

 \begin{equation*}
 \begin{split}
   \mathbb{L}_{\tilde{\omega}_{t_0}}f:&=\delta \delta D^-df \\
   &= \frac{1}{2}\Delta^2_{\tilde{\omega}_{t_0}}f+g_{\tilde{\omega}_{t_0}}\big(dd^cf,Ric(\varphi_{t_0})\big) + \frac{1}{2}g_{\tilde{\omega}_{t_0}}\big(df,dScal(\varphi_{t_0})\big),
  \end{split} 
 \end{equation*}

\noindent

\noindent where $D^-d$ and $\delta$ is introduced in $(\ref{define-lichne-pondéré})$ and  and $\mathbb{H}^{\tilde{\chi}}_{\tilde{\omega}_{t_0}}$ is introduced in (\ref{Hashimoto-operator-definition}). Since $\mathcal{L}_{\tilde{\omega}_{t_0}}$ is a self-adjoint operator with respect to $\tilde{\omega}_{t_0}^{[n]}$ we get

\begin{equation*}
    D (\pi \circ R)(\varphi_{t_0},t_0)[f,s]=\mathcal{L}_{\tilde{\omega}_{t_0}}f.
\end{equation*}

\noindent By Proposition \ref{decomposition-Lichne-prop} and Proposition \ref{decomposition-Hashimoto}, the restriction of $\mathcal{L}_{\tilde{\omega}_{t_0}}$  to $\mathcal{C}^{\infty}(V)^{\mathbb{T}}$ is equal to $ \mathcal{L}^V_{\omega_{t_0},\v}$, where

\begin{equation}{\label{decomposition-our-operator-prop}}
   \mathcal{L}^V_{\omega_{t_0},\v}:= -2t\mathbb{L}^V_{\omega_{t_0},\v}+(1-t)\mathbb{H}^{\chi}_{\omega_{t_0},\v}.
\end{equation}

\noindent In the above equality, $\omega_{t_0}$ is the K\"ahler metric on $(V,J_V)$ corresponding to $\tilde{\omega}_{t_0}$. By Proposition \ref{decomposition-Lichne-prop} and Proposition \ref{decomposition-Hashimoto} we obtain

\begin{equation}{\label{decomposition-notre-operator}}
\begin{split}
\mathcal{L}_{\tilde{\omega}_{t_0}}f=&\mathcal{L}^V_{\omega_{t_0},\v}f_s + t_0\mathbb{L}^S_xf_x+t_0\Delta^S_x\big(\Delta^V_{\omega_{t_0},\v}f_s\big) _x   \\
&+ t_0\Delta^V_{\omega_{t_0},\v} \big(\Delta_x^S f_x\big) _s+\sum_{a=1}^kU_a(x)\Delta_{a}f_x
\end{split}
\end{equation}

\noindent for all $f \in \mathcal{C}^{\infty}(M)^{\mathbb{T}}$, where $U_a(x)$ is a smooth function on $V$. By \cite[Lemma 3.1]{WH} the operator $\mathcal{L}_{\tilde{\omega}_{t_0}}$ extends to an isomorphism between H\"older spaces

\begin{equation}{\label{operator-continuity-path}}
    \mathcal{L}_{\tilde{\omega}_{t_0}} : \mathcal{C}^{4,\alpha}(M, \tilde{\omega}_{t_0})^{\mathbb{T}}\longrightarrow \mathcal{C}^{0,\alpha}(M, \tilde{\omega}_{t_0})^{\mathbb{T}},
\end{equation}

\noindent where  $\mathcal{C}^{4,\alpha}(M, \tilde{\omega}_{t_0})^{\mathbb{T}}$ is the space of $\mathbb{T}_M$-invariant functions with regularity $(4,\alpha)$ with zero mean value with respect to $\tilde{\omega}_{t_0}^{[n]}$ and similarly for $\mathcal{C}^{0,\alpha}(M, \tilde{\omega}_{t_0})^{\mathbb{T}}$. By (\ref{decomposition-our-operator-prop}), the restriction of the operator $\mathcal{L}_{\tilde{\omega}_{t_0}}$  to the space $\mathcal{C}^{4,\alpha}_{\v}(V,\omega_{t_0})^{\mathbb{T}}$ is equal to $\mathcal{L}_{\omega_{t_0},\v}^V$, where $\mathcal{C}^{4,\alpha}_{\v}(V,\omega_{t_0})^{\mathbb{T}}$ is the space of $\mathbb{T}_V$-invariant functions of regularity $(4,\alpha)$ with zero mean value with respect to $\v(m_{t_0})\omega^{[\ell]}_{t_0}$.

\begin{prop}{\label{operator-iso}}
 The operator  $\mathcal{L}^V_{\omega_{t_0},\v} : \mathcal{C}^{4,\alpha}_{\v}(V,\omega_{t_0})^{\mathbb{T}} \longrightarrow \mathcal{C}^{0,\alpha}_{\v}(V,\omega_{t_0})^{\mathbb{T}}$ is an isomorphism.
\end{prop}

\begin{proof}

Since $\mathcal{L}^V_{\omega_{t_0},\v}$ is the restriction of an injective operator, it is enough to prove the surjectivity. We proceed analogously to the proof of \cite[Lemma 8]{VA3}. 

We denote by $L^2_ {0,\v}(V)^{\mathbb{T}}$ the completion for the $L^2$-norm of $\mathcal{C}_{\v}^{0,\alpha}(V,\omega_{t_0})^{\mathbb{T}}$. We argue by contradiction. Assume $\mathcal{L}^V_{\omega_{t_0},\v} :\mathcal{C}_{\v}^{4,\alpha}(V,\omega_{t_0})^{\mathbb{T}} \rightarrow \mathcal{C}_{\v}^{0,\alpha}(V,\omega_{t_0})^{\mathbb{T}}$ is not surjective. Then, there exists $\phi \in  L^2_{0,\v}(V)^{\mathbb{T}}$ satisfying

\begin{equation}{\label{hypothese}}
    \int_V \mathcal{L}^V_{\omega_{t_0},\v}(f) \phi \v(m_{t_0}) \omega_{t_0}^{[\ell]}=0
\end{equation}

\noindent for all $f \in \mathcal{C}_{\v}^{4,\alpha}(V,\omega_{t_0})^{\mathbb{T}}$. We claim it implies

\begin{equation}{\label{Hypothesis2}}
        \int_M \mathcal{L}_{\tilde{\omega}_{t_0}}(f) \phi \tilde{\omega}_{t_0}^{[n]}=0
\end{equation}

\noindent for all $f \in \mathcal{C}^{4,\alpha}(M, \tilde{\omega}_{t_0})^{\mathbb{T}}$, which contradicts the surjectivity of $ \mathcal{L}_{\tilde{\omega}_{t_0}} : \mathcal{C}^{4,\alpha}(M, \tilde{\omega}_{t_0})^{\mathbb{T}} \longrightarrow \mathcal{C}^{0,\alpha}(M, \tilde{\omega}_{t_0})^{\mathbb{T}} $ established in \cite[Lemma 3.1]{WH}. Therefore, it is sufficient to show "$(\ref{hypothese}) \Rightarrow (\ref{Hypothesis2})$". For this we argue similar to the proof of \cite[Lemma 8]{VA3} using that the image of $\mathcal{L}^V_{\omega_{t_0},\v}$  is $L^2$-orthogonal to the subspace of constant functions with respect to $\v(m_{t_0})\omega^{[\ell]}_{t_0}$.

\end{proof}

\noindent By the Implicit Function Theorem applied to $\mathcal{L}^V_{\omega_{t_0},\v} : \mathcal{C}^{4,\alpha}_{\v}(V,\omega_{t_0})^{\mathbb{T}} \longrightarrow \mathcal{C}^{0,\alpha}_{\v}(V,\omega_{t_0})^{\mathbb{T}}$, we get a sequence of solutions $\{\varphi_{t_i}\}_{i\in \mathbb{N}}$ of (\ref{continuity-path-weithed}) of regularity $\mathcal{C}^{4,\alpha}$. By a well-known bootstrapping argument,  any solution of (\ref{continuity-path-weithed}), of regularity $\mathcal{C}^{4,\alpha}$ is in fact smooth. This concludes the proof of Proposition \ref{openess}.

\subsection{Closedness}

\begin{prop}
$S_{1}$ is closed
\end{prop}

\begin{proof}

By hypothesis, there exists an extremal K\"ahler metric in $[\tilde{\omega}_0]$. By Theorem \ref{Chen--Cheng-existence}, the relative Mabuchi energy $\mathcal{M}^{T}$ is $d_{1,T^{\mathbb{C}}}$-proper. Let $\{\varphi_i\}_{i \in \mathbb{N}} \subset \mathcal{K}(V,\omega_0)^{\mathbb{T}}$  be a sequence of solutions of (\ref{continuity-path-weithed}) given by Proposition \ref{openess} with $t_i \rightarrow t_1 < 1$. By Corollary \ref{change-of-metric-cor}, the sequence $\{\varphi_i\}_{i \in \mathbb{N}}$ lies in $\mathcal{K}(M,\tilde{\omega}_0)^{T}$. Consequently, the same argument as in \cite[Lemma~3.3]{WH} shows the existence of a smooth limit $\varphi_{t_1} \in \mathcal{K}(M,\tilde{\omega}_0)^{T}$. Moreover, it follows from (\ref{exact-sequence}) and (\ref{identification-functio-cartesian}),  that  $\mathcal{K}(V,\omega_0)^{\mathbb{T}}$ is closed in $\mathcal{K}(M,\tilde{\omega}_0)^{T}$ for the Frechet topology. Then, the K\"ahler potential limit $\varphi_{t_1}$ belongs to $ \mathcal{K}(V,\omega)^{\mathbb{T}}$. In particular $\tilde{\omega}_{\varphi_{t_1}}$ is a compatible K\"ahler metric.

Let $\tilde{\varphi}_{t_i} \in \mathring{\mathcal{K}}(V,\omega_0)^{\mathbb{T}}$ (see $(\ref{normalized-compatible-potential})$) be the solution of $(\ref{continuity-path-weithed})$ at $t_i$ for $t_i$ increasing to $1$. By Theorem \ref{Chen--Cheng-existence}, $\mathcal{M}^{T}$ is $d_{1,T^{\mathbb{C}}}$-proper. Then, by Corollary \ref{change-of-metric-cor}, we get a bound with respect to $d_{1,T^{\mathbb{C}}}$, that is

\begin{equation*}
    \sup_{i\in \mathbb{N}}d_{1,T^{\mathbb{C}}}(0,\tilde{\varphi}_{t_i}) < \infty.
\end{equation*}

\noindent  By definition of $d_{1,T^{\mathbb{C}}}$, there exists $\gamma_i \in T^{\mathbb{C}}$ and $\varphi_{t_i} \in \mathring{\mathcal{K}}(M,\tilde{\omega}_0)^{T}$ such that $\omega_{\varphi_{t_i}}=\gamma_i^*\omega_{\tilde{\varphi}_{t_i}}$, and

\begin{equation*}
    \sup_{i\in \mathbb{N}}d_1(0,\varphi_{t_i}) < \infty.
\end{equation*}

 By definition $\gamma_i$ preserves $J$. Moreover, the form $\tilde{\omega}_{\varphi_{t_i}}$ is not compatible in general since the connection form $\theta$ and the base K\"ahler metrics $\omega_{a}$ may change by the action of $\gamma_i$. However, by Proposition \ref{extremal-vector-fields}, the $T^{\mathbb{C}}$-action commutes with the $\mathbb{T}_M$-action. Then, for each $t_i$, the $\mathbb{T}_M$-action is still rigid and semi-simple  (see Remark \ref{remark-semi-simple}). According to \cite{VA2}, $\tilde{\omega}_{\varphi_{t_i}}$ is given by the generalized Calabi ansatz, with a fixed stable quotient $S= \prod_{a=1}^k S_a$ with respect to the complexified action $\mathbb{T}_M^{\mathbb{C}}$. Thus,  there exists a connection 1-form $\theta_{Q, t_i}$ with curvature

\begin{equation*}
    d\theta_{t_i}=\sum_{a=1}^k \pi_S^*(\omega_{a,t_i}) \otimes p_{a,t_i} \text{ } \text{ } p_{a,t_i} \in \Lambda
\end{equation*}

\noindent such that $\tilde{\omega}_{\varphi_{t_i}}$ is given by

\begin{equation*}
   \tilde{\omega}_{\varphi_{t_i}}= \sum_{a=1}^k(\langle p_{a,t_i},m_{\varphi_{t_i}}\rangle +c_{a,t_i})\pi_S^*(\omega_{a,t_i}) + \langle dm_{\varphi_{t_i}} , \theta_{t_i} \rangle.
\end{equation*}

Since $\tilde{\omega}_{\varphi_{t_i}} \in [\tilde{\omega}_0]$, $c_{a,t_i}=c_a$ and $p_{a,t_i}=p_{a}$. By \cite[Theorem~3.5]{WH},  $\tilde{\omega}_{t_i}$ converge smoothly to an extremal metric $\omega_{\varphi_{1}}$. Furthermore, by Propositon \ref{extremal-vector-fields}, the extremal vector field $\xi_{ext}$ of $[\tilde{\omega}_0]$ relatif to $T$ is in the Lie algebra $\mathfrak{t}$ of $\mathbb{T}_M$. Then, by Corollary \ref{equivalence-scalar} and the smooth convergence of $\tilde{\omega}_{\varphi_{t_i}}$ to $\tilde{\omega}_{\varphi_1}$, we get

\begin{equation}{\label{extremal-presque}}
    \langle m_{\varphi_1}, \xi_{ext} \rangle + c_{ext}=\sum_{a =1}^k \frac{Scal(\omega_{a,1})}{\langle p_a,m_{\varphi_1} \rangle +c_a} + \frac{1}{\mathrm{v}(m_{\varphi_1})} Scal_{\v}(\omega_{\varphi_1}),
\end{equation}

\noindent where $\omega_{\varphi_1}$ is the K\"ahler metric on $(V,J_V)$ corresponding to $\tilde{\omega}_{\varphi_1}$. Taking the exterior differential $d_{S_a}$ on $S_a$ in (\ref{extremal-presque}) we get $d_{S_a}Scal(\omega_{a,1})=0$ for all $1\leq a \leq k$, i.e. $\omega_{a,1}$ has constant scalar curvature. Yet, $[\omega_{a,1}]=[\omega_a]$, showing that $Scal(\omega_{a,1})=Scal_a$. By definition of $\w \in \mathcal{C}^{\infty}(P,\R)$,  we get

\begin{equation*}
 Scal_{\v}(\omega_{\varphi_1})=\w(m_{\varphi_1}).
\end{equation*}

\end{proof}

\begin{corollary}
In a compatible K\"ahler class, the extremal metrics are given by the Calabi ansatz of \cite{VA2}. Equivalently,  in a compatible K\"ahler class, the extremal metrics are induced by $(\v,\w)$-cscK metrics on $(V,J_V)$ via $(\ref{metriccalabidata})$ for a suitable connection $\theta$ and suitable K\"ahler metric $\omega_a$.
\end{corollary}

\begin{proof}
Suppose there exists an extremal metric $\omega_{1}$ in $[\tilde{\omega}]$. By a  result of Calabi \cite{EC}, $\omega_{1}$ is  invariant by some maximal torus $T \subset \aut_{\red}(M)$. Conjugating if necessary, we can assume that $\mathbb{T}_M \subset T$.
By Theorem \ref{theoremA}, there exists a compatible extremal metric $\omega_{2}$ in $[\tilde{\omega}]$. By Lemma \ref{change-of-metric-cor}, $\omega_{2}$ is $T$-invariant. Then, by unicity of extremal K\"ahler metrics invariant by a maximal torus of the reduced automorphism group \cite{BB, CPZ, AL}, there exists $\gamma \in T^{\mathbb{C}}$ such that $ \omega_{1}=\gamma^*\omega_{2}$. Since  $\mathbb{T}_M \subset T$, the action of $\mathbb{T}$ on $(M,J, \omega_{1})$ is still rigid and semi-simple, see Remark \ref{remark-semi-simple}. Thus, according to \cite{VA2}, $\omega_{1}$ is given by the Calabi ansatz.

\end{proof}

\section{Weighted toric K-stability}{\label{section-toric}}

\subsection{Complex and symplectic points of view}{\label{subsection-complex-vs-stmplectic}}
In view of Section \ref{section-application}, where we will
consider almost K\"ahler structures on toric varieties, we briefly recall the well-known correspondence between
symplectic and K\"ahler potentials of toric K\"ahler manifolds, established and widely used over the years, notably in \cite{VA4, VA2, SKD, VG}.   We use the notation and conventions of \cite{VA1}, which differ in places from those used in  \cite{MA2, SKD}.

Let $(V,\omega_0,\mathbb{T})$ be a toric symplectic manifold classified by its \textit{labelled integral Delzant polytope} $(\PL)$ \cite{VA4, TD}, where $\textbf{L} =(L_j)_{j=1\dots k}$ is the collection of non-negative defining affine-linear functions for $P$, with $dL_j$ being primitive elements of the lattice $\Lambda$ of circle subgroups of $\mathbb{T}$. Choose a K\"ahler structure $(g,J)$ on $(V,\omega_0,\mathbb{T})$ and denote by $(m_0,t_J)$ the associated moment map, i.e. $m_0 : V \longrightarrow \mathfrak{t}^*$ is the moment map of $(V,\omega_0,\mathbb{T})$ and $t_J : V^0 \longrightarrow \mathfrak{t}/2\pi \Lambda $ is the angular coordinates (unique modulo an additive constant) depending on the complex structure $J$ (see \cite[Remark 3]{VA2}). These coordinates are symplectic, i.e. $\omega_0$ is given by (\ref{toric_symplectic}) for its respective moment-angular coordinates. The K\"ahler structure $(g,J)$ is defined on $V^0$ by a smooth strictly convex function $u$ on $P^0$ via

\begin{equation}{\label{metric-toric2}}
    g = \langle dm_0, \textbf{G}, dm_0 \rangle + \langle dt_J, \textbf{H},dt_J \rangle \text{ } \text{ and } \text{ } J dm_0= \langle \textbf{H} , dt_J \rangle,
\end{equation}

\noindent where $\textbf{G}:=\textnormal{Hess}(u)$ is a  positive definite $S^2\mathfrak{t}$-valued function and $\textbf{H}$ is $S^2\mathfrak{t}^*$-valued function on $P^0$ and  inverse of $\textbf{H}$ (when seen $\textbf{H} : \mathfrak{t} \longrightarrow \mathfrak{t}^*$ and $\textbf{G}: \mathfrak{t}^* \longrightarrow \mathfrak{t}$ in each point in $P^0$) and $\langle \cdot , \cdot , \cdot \rangle$ denote the point wise contraction $\mathfrak{t}^* \times S^2 \mathfrak{t} \times \mathfrak{t}^*$ or its dual. It is shown in \cite[Lemma 3]{VA2}, that for two $\mathbb{T}$-invariant K\"ahler structures on $(V,\omega_0,\mathbb{T})$, given explicitly on $V^0$ by (\ref{metric-toric2}) with the same matrix $\textbf{H}$, there exists a $\mathbb{T}$-equivariant K\"ahler isomorphism between them.

Conversely, smooth strictly convex functions $u$ on $P^0$ define $\mathbb{T}$-invariant $\omega_0$-compatible K\"ahler structures on $V^0$ via $(\ref{metric-toric2})$. The following Proposition established in \cite{VA2} gives a criterion for the metric to compactify.

\begin{prop}{\label{boudary}}
Let $(V,\omega, \mathbb{T})$ be a compact toric symplectic $2\ell$-manifold with momentum map $m_{\omega} : V \rightarrow P$ and $u$ be a smooth strictly convex function on $P^0$. Then the positive definite  $S^2\mathfrak{t}^{*}$-valued function $\textnormal{\textbf{H}}:=\textnormal{Hess}(u)^{-1}$ on $P^0$   comes from a $\mathbb{T}$-invariant, $\omega$-compatible  K\"ahler metric $g$ via $(\ref{metric-toric2})$ if and only if it satisfies the following conditions:

\begin{itemize}
    \item \textnormal{[smoothness]} $\textnormal{\textbf{H}}$ is the restriction to $P^0$ of a smooth $S^2\mathfrak{t}^{*}$-valued function on $P$;
 \item \textnormal{[boundary values]} for any point $y$ on the codimension one face $F_j \subset P$   with
inward normal $u_j$, we have

\begin{equation}{\label{bounday(condition-equation}}
 \textnormal{\textbf{H}}_y(u_j , \cdot ) = 0 \text{ and } (d\textnormal{\textbf{H}})_y(u_j , u_j ) = 2u_j,
\end{equation}

\noindent where the differential $d\textnormal{\textbf{H}}$ is viewed as a smooth $S^2\mathfrak{t}^*\otimes \mathfrak{t}$-valued function on $P$;
\item \textnormal{[positivity]} for any point y in the interior of a face $F \subseteq P$, $\textnormal{\textbf{H}}_y(\cdot,\cdot)$ is positive definite
when viewed as a smooth function with values in $S^2(\mathfrak{t}/\mathfrak{t}_F )^*$, where $\mathfrak{t}_F \subset \mathfrak{t}$ the vector subspace spanned by the
inward normals $u_j$ in $\mathfrak{t}$ to the codimension one faces $F$.

\end{itemize}
\end{prop}

\begin{define}
Let $\mathcal{S}(\PL)$ be the space of smooth strictly convex functions on the interior of $P^0$ such that $\textnormal{\textbf{H}} =\textnormal{Hess}(u)^{-1}$ satisfies the conditions of Proposition $\ref{boudary}$.
\end{define}

\begin{remark}
In the above Proposition and Definition, we use as a model metric the Guillemin K\"ahler metric $(g_0, J_0)$  \cite{VG}, given by $(\ref{metric-toric2})$ for the symplectic potential \[ u_0 := \frac{1}{2} \sum_{j=1}^d L_j \log L_j,\] where $L_j, j=1, \ldots, d$ are the affine-linear functions defining the polytope. This introduces a discrepancy of a factor $1/2$ with respect to the normalization used in \cite{SKD}, which in turn will result in some obvious modification of the formula for the (weighted) Futaki invariant in Section $\ref{subsection-Futaki}$.
\end{remark}

Thus, there exists a bijection between $\mathbb{T}$-equivariant isometry classes of $\mathbb{T}$-invariant $\omega_0$-compatible K\"ahler structures and smooth functions $\textbf{H}= \textnormal{Hess}(u)^{-1}$, where $u \in \mathcal{S}(\PL)$.

 We fix the Guillemin K\"ahler structure $J_0$ on $(V,\omega_0,\mathbb{T})$. Consider another $\omega_0$-compatible K\"ahler structure $J_u$ defined by a symplectic potential $u \in S(\PL)$ via (\ref{metric-toric2}). Donaldson shows that \cite{SKD2} there is a biholomorphism  $\Phi_u : (V,J_u) \cong (V,J_0)$.  Let $\omega_u := \Phi_u^*(\omega_0)$ and $\phi_u$ and $\phi_{u_0}$ be the Legendre transforms of $u$ and $u_0$, respectively. By \cite{VG}, we have that

\begin{equation}{\label{relation-potential0}}
     \omega_u = \omega_0 + dd^c_{J_0} \varphi_u, \qquad \varphi_u(y_0):= \phi_u(y_0)- \phi_{u_0}(y_0),
\end{equation} 

\noindent where $y_0 = \nabla u_0$ are the pluriharmonic coordinates with respect to $J_0$.

Conversely, using the dual Legendre transform, any $\mathbb{T}$-invariant K\"ahler potential $\varphi \in \mathcal{K}(V,\omega_0)^{\mathbb{T}}$ gives rise to a symplectic potential $u \in \mathcal{S}(\PL)$   through $(\ref{relation-potential0})$.  The key point of this correspondence is that, as show in \cite{DG}, the path $\varphi_{u_t} \in \mathcal{K}(V,\omega_0)^{\mathbb{T}}$ corresponding to a path $u_t \in \mathcal{S}(\PL)$ satisfies

\begin{equation}{\label{relation-potentials}}
    \frac{d}{dt}u_t=-\frac{d}{dt}\varphi_{t}.
\end{equation}

\subsection{Generalized Abreu's equation}

Thanks to Abreu \cite{MA2}, the scalar curvature $Scal(u)$ associated to a symplectic potential $u\in \mathcal{S}(\PL)$ is expressed by

\begin{equation}{\label{abreu}}
    Scal(u)=\sum^{\ell}_{i,j=1}-(H^u_{ij})_{,ij},
\end{equation}

\noindent where the partial derivatives and the inverse Hessian $(H_{ij}^u)=\textnormal{Hess}(u)^{-1}$ of $u$  is taken in a fixed basis $\boldsymbol{\xi}^*$ of $\mathfrak{t}^*$. 

From \cite[Sect.~6]{AL} and the computation of \cite[Sect. 3]{VA5}, the $\v$-scalar curvature associated to a symplectic potential $u\in \mathcal{S}(\PL)$ and a positive weight function $\v$ is given by

\begin{equation}{\label{v-scalar-curv-toric}}
  Scal_{\mathrm{v}}(u)= - \sum_{i,j=1}^{\ell}(\v H^u_{ij})_{,ij}.
\end{equation}

Let  $\v \in \mathcal{C}^{\infty}(P,\R_{>0})$ and  $\w \in \mathcal{C}^{\infty}(P,\R)$. According to Definition \ref{define-scalv}, a K\"ahler structure $(J_u,g_u)$ on $(V,\omega_0,\mathbb{T})$, associated with a symplectic potential $u \in \mathcal{S}(\PL)$, is $(\v,\w)$-cscK if and only if it satisfies

\begin{equation}{\label{equation-Abreu}}
    - \sum_{i,j=1}^{\ell}(\v H^u_{ij})_{,ij}=\w.
\end{equation}

This formula generalizes the expression $(\ref{abreu})$ and is referred to as \textit{the generalized Abreu equation}. This equation has been studied for example in \cite{VA3, LLS, LLS2, LLS3}.

\subsection{Weighted Donaldson--Futaki invariant}{\label{subsection-Futaki}}

Following \cite{SKD, AL, LLS}, for $\v \in \mathcal{C}^{\infty}(P,\R_{>0})$ and  $\w \in \mathcal{C}^{\infty}(P,\R)$, we introduce the $(\v,\w)$-Donaldson--Futaki invariant

\begin{equation}{\label{define-futaki}}
  \mathcal{F}_{\v,\w}(f):=  2\int_{\partial P} f\v d\sigma -  \int_P  f  \w  dx,
\end{equation}

\noindent for all continuous functions $f$ on $P$, where $d\sigma$ is the induced measure on each face $F_i \subset \partial P$ by letting $dL_i \wedge d \sigma = -dx$, where $dx$ is the Lesbegue measure on $P$.

\begin{convention}
The weights $\v>0$ and $\w \in C^{\infty}(P,\R)$ satisfy

\begin{equation}{\label{annulation-Futaki}}
    \mathcal{F}_{\v,\w}(f)=0 
\end{equation}
\noindent for all $f$ affine-linear on $P$.
\end{convention}

Integration by parts (see e.g. \cite{SKD}) reveals that ({\ref{annulation-Futaki}) is a necessary condition for the existence of $(\v,\w)$-cscK metric on $(V,\omega_0,\mathbb{T})$.

\begin{remark}
Notice that in the case of semi-simple principal toric fibrations, the weights given by $(\ref{weights})$ satisfy $(\ref{annulation-Futaki})$ above.
\end{remark}
 
 \subsection{The weighted  Mabuchi energy}

 The volume form $\omega_0^{[\ell]}$ on $V$ is pushed forward to the measure $dx$ via the moment map $m_{0}$. Seen as functional on $\mathcal{S}(\PL)$ via (\ref{relation-potentials}), the weighted Mabuchi energy $\mathcal{M}_{\v,\w}$ satisfies

\begin{equation*}
	d_{u}\mathcal{M}_{\mathrm{v},\mathrm{w} }(\Dot{u}) =  \int_P\bigg( - \sum^{\ell}_{i,j=1}\big(\v H^u_{ij}\big)_{,ij}  - \w \bigg)\Dot{u}dx.
\end{equation*}

\noindent  From  \cite[Lemma 6]{VA5} (see also Lemma \ref{IPP-lemma} below) we get

\begin{equation*}
    d_{u}\mathcal{M}_{\mathrm{v},\mathrm{w} }(\Dot{u})= \mathcal{F}_{\v,\w}(\dot{u}) -  \int_P \sum_{i,j=1}^{\ell} \v H_{ij}^u \dot{u}_{,ij}dx,
\end{equation*}

\noindent where $\mathcal{F}_{\v,\w}$ is the Donaldson--Futaki invariant defined in (\ref{define-futaki}). Using that the derivative of $\textnormal{tr}\textbf{H}^{-1}d\textbf{H}$ is  $d \log \det \textbf{H}$, we get

\begin{equation*}
    \mathcal{M}_{\v,\w}(u)=\mathcal{F}_{\v,\w}(u)- \int_P \log \det \textnormal{Hess}(u)\textnormal{Hess}(u_0)^{-1} \v dx.
\end{equation*}

We denote by $\mathcal{CV}^{\infty}(P)$ the set of continuous convex functions on $P$ which are smooth in the interior $P^0$. Using the same argument than  \cite[Lemma 3.3.5]{SKD}, since $\v$ is smooth and $P$ compact, we get:

 \begin{lemma}{\label{IPP-lemma}}
Let $\textnormal{\textbf{H}}$ be any smooth $S^2\mathfrak{t}^*$-valued function on $P$ which satisfies the boundary condition $(\ref{bounday(condition-equation})$ of Proposition \ref{boudary}, but not necessarily the positivity
condition. For any $\v \in \mathcal{C}^{\infty}(P,\R_{>0})$ and $f \in \mathcal{CV}^{\infty}(P)$:

\begin{equation}\label{equation-IPP}
  \int_P \sum_{i,j=1}^{\ell}\big(\v H_{ij}\big)f_{,ij} dx =  \int_P \bigg(\sum_{i,j=1}^{\ell}\big(\v H_{ij}\big)_{,ij}\bigg)f dx + 2 \int_{\partial P} f \v d\sigma.
\end{equation}

\noindent In particular, $\int_P \sum_{i,j=1}^{\ell}\big(\v H_{ij}\big)f_{,ij}dx <  \infty$.

\end{lemma}

\noindent The following result and proof are generalizations of \cite[Proposition 3.3.4]{SKD}.

\begin{prop}{\label{extension}}
Let $\v \in \mathcal{C}^{\infty}(P,\R_{>0})$ and $\w\in \mathcal{C}^{\infty}(P,\R)$. The Mabuchi energy $\mathcal{M}_{\v,\w}$ extends to the set $\mathcal{CV}^{\infty}(P)$ as functional with values in $(- \infty, + \infty]$. Moreover, if there exists $u\in \mathcal{S}(\PL)$ corresponding to a $(\v,\w)$-cscK metric, i.e. which satisfies $(\ref{equation-Abreu})$, then $u$ realizes the minimum of $\mathcal{M}_{\v,\w}$ on $\mathcal{CV}^{\infty}(P)$.
\end{prop}

\begin{proof}
The linear term $\mathcal{F}_{\v,\w}$ is well-defined on $\mathcal{CV}^{\infty}(P)$. We then focus on the nonlinear term of $\mathcal{M}_{\v,\w}$. Let $u\in \mathcal{S}(\PL)$ and $h\in \mathcal{CV}^{\infty}(P)$. Suppose $\det \text{Hess}(h)\neq 0$. By convexity of the functional $-\log \det$  on the space of positive definite matrices, we get:

\begin{equation*}
        - \log \det \text{Hess}(h)+ \log \det \text{Hess}(u) \geq - \text{Tr}\big(\text{Hess}(u)^{-1}(\text{Hess}(f)\big),
\end{equation*}

\noindent where $f=h-u$. Turning this around and multiplying by $\v$, we obtain:

\begin{equation*}
         \v \log \det \text{Hess}(h) \leq \v \log \det \text{Hess}(u) +  \v \text{Tr}\big(\text{Hess}(u)^{-1}(\text{Hess}(f)\big).
\end{equation*}

\noindent By linearity of (\ref{equation-IPP}),  the equality still holds when we replace $f$ by a difference of two functions in $\mathcal{CV}^{\infty}(P)$. In particular, this shows that $ \v \text{Tr}\big(\text{Hess}^{-1}(u)(\text{Hess}(f)\big)$ is integrable on $P$ and hence, by the previous inequality, $\v \log \det \text{Hess}(h)$ is integrable too. Now if the determinant of $h$ is equal to $0$, we define the value of $\mathcal{M}_{\v,\w}(h)$ to be $+\infty$. Then $\mathcal{M}_{\v,\w}$ is well-defined on $\mathcal{CV}^{\infty}(P)$. Suppose $u$ satisfies (\ref{equation-Abreu}). If $\det\textnormal{Hess}(f)=0$, then we trivially get $\mathcal{M}_{\v,\w}(u) \leq \mathcal{M}_{\v,\w}(f)$. Now, admit $\det\textnormal{Hess}(f)\neq 0$ and consider the function $g(t)=\mathcal{M}_{\v,\w}(u+tf)$. The function $g$ is therefore a convex function. Moreover, $g$ is differentiable at $t=0$ with

\begin{equation*}
g'(0)=-\bigintsss_P\left( \sum^{\ell}_{i,j=1}(\v H^u_{ij})_{,ij}  - \w\right)fdx,
\end{equation*}

\noindent which is equal to $0$ by hypothesis on $u$. Then $\mathcal{M}_{\v,\w}(u) \leq \mathcal{M}_{\v,\w}(f) $ by the convexity of $g$.

\end{proof}

\subsection{Properness and $(\v,\w)$-uniform K-Stability }{\label{subsection-K-stability-and-proper}}

\noindent Following \cite{SKD, GS} (see also \cite[Chapter 3.6]{VA1}), we fix $x_0 \in P^0$ and consider the following  normalization
 
 \begin{equation}{\label{normalized-function-polytope}}
     \mathcal{CV}^{\infty}_*(P):=\{ f \in \mathcal{CV}^{\infty}(P) \text{ } | \text{ } f(x) \geq f(x_0)=0 \}.
 \end{equation}
 
\noindent Then, any $f\in \mathcal{CV}^{\infty}(P)$ can
be written uniquely as $f = f^*+f_0$, where $f_0$ is affine-linear and $f^*=\pi(f) \in \mathcal{CV}_*^{\infty}(P)$, where $\pi : \mathcal{CV}^{\infty}(P) \longrightarrow \mathcal{CV}^{\infty}_*(P) $ is the linear projection.

\begin{define}{\label{uniform-K-stable}}
A Delzant polytope $(\PL)$ is  $(\v,\w)$-uniformly K-stable if there exists $\lambda > 0$ such that

\begin{eqnarray}{\label{uniform-equation}}
\mathcal{F}_{\v,\w}(f) \geq \lambda \| f^*\|_{1}
\end{eqnarray}

\noindent for all  $f \in \mathcal{CV}^{\infty}(P)$, where $\| \cdot \|_{1}$ denotes the $L^1$-norm on $P$.

\end{define}

\begin{prop}{\label{stable-equivaut-energy-propr-v}}
Suppose $(\PL)$ is $(\v,\w)$-uniformly K-stable. Then there exists  $C>0$ and $D \in \R $ such that
    
\begin{equation*}
    \mathcal{M}_{\v,\w}(u) \geq C \|u^*\|_{1} + D
\end{equation*}

\noindent for all $u \in \mathcal{S}(\PL)$.
\end{prop}

\begin{proof}

 This result when $\v=1$ is due to  \cite{SKD, ZZ}. The proof is an adaptation of the exposition in \cite{VA1}.

Let $u_0 \in \mathcal{S}(\PL)$ be the Guillemin K\"ahler potential. We consider $\mathcal{F}_{\v, \w_{0}}$ where $\w_{0}:=Scal_{\v}(u_0)$.  For any $f\in \mathcal{CV}_*^{\infty}(P)$ there exists $C>0$ such that 

\begin{eqnarray*}
| \mathcal{F}_{\v,\w}(f) - \mathcal{F}_{\v,\w_0}(f) | \leq 2C \|f \|_{1}.
\end{eqnarray*}

\noindent  Since $(\PL)$ is $(\v,\w)$-uniformly stable  we get

\begin{eqnarray*}
| \mathcal{F}_{\v,\w_0}(f) - \mathcal{F}_{\v,\w}(f) | \leq C_1 \mathcal{F}_{\v,\w}(f)  - C\|f \|_{1}, 
\end{eqnarray*}

\noindent where $C_1$ is a positive constant depending (\ref{uniform-equation}). We deduce that

\begin{equation}{\label{inequality}}
\mathcal{F}_{\v,\w_0}(f)  \leq  \tilde{C}\mathcal{F}_{\v,\w}(f) - C\|f \|_{1},
\end{equation}

\noindent  where $\tilde{C}:=C_1+1$. 
\noindent By Proposition \ref{extension}, the Mabuchi energy extends to $\mathcal{CV}_*^{\infty}(P)$. Then, by (\ref{inequality}) and (\ref{annulation-Futaki}), for any $u\in\mathcal{S}(\PL)$,

\begin{eqnarray*}
\mathcal{M}_{\v,\w}(u) &=&  \mathcal{F}_{\v,\w}(u^*)-\int_{P}\v \log\det\text{Hess}(u^*) \text{Hess}(u_0)^{-1} dx  \\
&\geq&  \tilde{C} \mathcal{F}_{\v,\w_{0}}(u^*) + C \| u^* \|_{1}  -\int_{P}\v \log\det\text{Hess}(u^*) \text{Hess}(u_0)^{-1} dx \\
&=& \mathcal{M}_{\v,\w_{0}}(\tilde{C} u^*) + \int_{P}\v \log\det\text{Hess}(\tilde{C} u^*) \text{Hess}(u^*)^{-1} dx    + C \| u^* \|_{1} \\
&= & \mathcal{M}_{\v,\w_{0}}(\tilde{C} u^*) + n \log \tilde{C}  \int_P \v dx  + C \| u^* \|_{1}.
\end{eqnarray*}

\noindent The Mabuchi energy $\mathcal{M}_{\v,\w_0}$ reaches its minimum at the potential $u_0\in \mathcal{S}(\PL)$, which is solution of

\begin{equation*}
    Scal_{\v}(u_0)=\w_0.
\end{equation*}

\noindent  In particular, $\mathcal{M}_{\v,\w_0}$ is bounded from below on $\mathcal{CV}^{\infty}(P)$ by Proposition $\ref{extension}$. Letting $D:=\inf_{\mathcal{CV}^{\infty}} \mathcal{M}_{\v,\w_0}+n\log \tilde{C} \int_P \v dx$ we get the result.

\end{proof}

\subsection{Existence of $(\v,\w)$-cscK is equivalent to $(\v,\w)$-uniform K-stability}

\noindent The following is established in \cite[Theorem 2.1]{LLS2} and is due to \cite{CLS} when $\v=1$.

\begin{prop}{\label{existence-implies-stable}}
Suppose there exists an $(\v,\w)$-cscK metric in $(V,[\omega_0], \mathbb{T})$, i.e. $(\ref{equation-Abreu})$ admits a solution $u \in \mathcal{S}(\PL)$. Then $P$ is $(\v,\w)$-uniformly K-stable. 
\end{prop}

We now focus on the converse. We consider the space of normalized K\"ahler potentials $(\ref{normalized-compatible-potential})$ and normalized symplectic potentials

\begin{equation}{\label{normalized-smyplecic-potenial}}
    \mathring{\mathcal{S}}_{\v}(\PL):=\{ u \in \mathcal{S}(\PL) \text{ } | \text{ }\int_P u \v dx = \int_P u_0 \v  dx \}.
\end{equation}

\begin{lemma}{\label{correspondance}}
For any $\mathring{u}_t \in \mathring{\mathcal{S}}_{\v}(\PL)$, the corresponding K\"ahler potential $\varphi_t=\varphi_{\mathring{u}_t}$ obtained via $(\ref{relation-potential0})$ belongs to the space of normalized K\"ahler potential $\mathring{\mathcal{K}}_{\v}(V,\omega_0)^{\mathbb{T}}$ defined in $(\ref{normalized-compatible-potential})$. Conversely, any path in $\mathring{\mathcal{K}}_{\v}(V,\omega_0)^{\mathbb{T}}$ comes from a path $\mathring{u}_t$ in $\mathring{\mathcal{S}}_{\v}(\PL)$.

\end{lemma}

\begin{proof}
By \cite[Lemma 2.4]{BW}, the functional $\mathcal{I}_{\v}$ defined in  (\ref{Ir-functionnal}), is also characterized by its variation for general weights $\v \in \mathcal{C}^{\infty}(P,\R_{>0})$. Then, a path $\varphi_t \in \mathcal{K}_{\v}(V,\omega_0)^{\mathbb{T}}$ starting from $0$ belongs to $\mathring{\mathcal{K}}_{\v}(V,\omega_0)^{\mathbb{T}}$ if and only if

 \begin{equation}{\label{annulation-ir}}
     \int_V \dot{\varphi}_t \v(m_{\varphi_t})\omega_t^{[\ell]} =0
 \end{equation}
 
\noindent for all $\dot{\varphi}_t \in T_{\varphi_t}\mathcal{K}_{\v}(V,\omega_0)^{\mathbb{T}}$. By pushing-forward the measure $\omega_t^{[\ell]}$ via $m_{\omega_{\varphi_t}}$ and using $(\ref{relation-potentials})$ we get that $(\ref{annulation-ir})$ is equivalent to

\begin{equation*}
    \int_P \dot{u}_t\v dx=0,
\end{equation*}

\noindent where $u_t$ is the path corresponding to $\varphi_t$ via $(\ref{relation-potential0})$. The conclusion follows from the convexity of $\mathcal{S}(\PL)$.
\end{proof}

\begin{theorem}{\label{theorem-B}}
Let $(M,J, \tilde{\omega}_0, \mathbb{T})$ be a semi-simple principal toric fibration with K\"ahler toric fiber $(V,J_V, \omega_0, \mathbb{T})$. Let $(\v,\w)$ be the weights defined in $(\ref{weights})$ and denote by $P$ the Delzant polytope associated to $(V, \omega_0, \mathbb{T})$. Then there exists a $ (\v,\w)$-weighted cscK metric in $[\omega_0]$ if and only if $P$ is $(\v,\w)$-uniformly $K$-stable. In particular, the latter condition is necessary and sufficient for $[\tilde{\omega}_0]$ to admit an extremal K\"ahler metric.
\end{theorem}

\begin{proof}
 Suppose there exists a $(\v,\w)$-cscK metric in $[\omega_0]$. By Proposition \ref{existence-implies-stable}, $P$ is $(\v,\w)$-uniformly K-stable.

Conversely, suppose $P$ is $(\v,\w)$-uniformly K-stable. We are going to show that there are uniform positive constants $\tilde{A}$ and $\tilde{B}$ such that 

\begin{equation}{\label{coercivity-conclu}}
    \mathcal{M}_{\v, \w}(\varphi) \ge\tilde{ A} \inf_{\gamma \in \mathbb{T}^{\mathbb{C}}}   d^V_{1,\v}(0, \gamma \cdot \varphi) - \tilde{B},
\end{equation}
 where $d^V_{1,\v}$ is defined in Lemma \ref{restriction-distance} and $\mathcal{M}_{\v,\w}$ is the weighted Mabuchi energy of the K\"ahler toric fiber $(V,J_V,[\omega_0], \mathbb{T})$, see (\ref{definition-weighted-eneergy}). For all $\varphi \in \mathring{\mathcal{K}}(V,\omega_0)^{\mathbb{T}}$, there exists $\gamma \in \mathbb{T}^{\mathbb{C}}$ such that the  symplectic potential $u_{\gamma \cdot \varphi}$
corresponding to $\gamma \cdot \varphi$ satisfies $d_{x_0}u_{\gamma \cdot \varphi}=0$. By  Lemma  $\ref{correspondance}$ and the inequality in \cite[(66)]{VA1}, we have

\begin{equation}{\label{coercivity-conclu2}}
d^V_{1,\mathbb{T}^{\mathbb{C}}}(0,\gamma \cdot \varphi)  \le A \int_{P}|u_{\varphi}^* - u_0^*| dx \le A\|u_{\varphi}^*\|_{1} + B,
\end{equation}

\noindent for some uniform constants $A>0$ and $ B>0$, where $d_{1,\mathbb{T}}^V$ is the $d_1$ distance relative to $\mathbb{T}^{\mathbb{C}}$ (see $(\ref{distance-relative})$) on $\mathcal{K}(V,\omega_0)^{\mathbb{T}}$, $u_{\varphi}\in S(\PL)$ is the symplectic potential corresponding to $\varphi$ and $u^*_{\varphi}$ is its normalization in $S(\PL) \cap  \mathcal{CV}^{\infty}_*(P)$, see (\ref{normalized-function-polytope}).  Since $\v>0$ on $P$, we have for the weighted distance $ d_{1,\v}^V \le C d_1^V$. Then (\ref{coercivity-conclu}) follows from  (\ref{coercivity-conclu2}) and Proposition \ref{stable-equivaut-energy-propr-v}.

Let $T$ be the maximal torus in $\mathrm{Aut}_{\red}(M)$ containing $\mathbb{T}_M$ and satisfying $(\ref{exact-sequence})$. By Lemma  \ref{mabuchi-energy-restriction}, and our choice of normalization $(\ref{normalized-compatible-potential})$, the Mabuchi energy $\mathcal{M}^{T}$ restricts to $\mathcal{M}_{\v,\w}$ on $ \mathring{\mathcal{K}}_{\v}(V,\omega_0)^{\mathbb{T}}$. We denote by $d_{1,T^{\mathbb{C}}}$ the $d_1$ distance relative to $T^{\mathbb{C}}$ on $\mathcal{K}(M,\tilde{\omega}_0)^T$. Since any $\mathbb{T}^{\mathbb{C}}_M$-orbit lies in a $T^{\mathbb{C}}$-orbit, by (\ref{coercivity-conclu}) and Lemma \ref{restriction-distance},  $\mathcal{M}^{T}$ is $d_{1,T^{\mathbb{C}}}$-proper on $\mathring{\mathcal{K}}_{\v}(V,\omega_0)^{\mathbb{T}}$ in the sense of Definition \ref{def-proper}. 

In the proof of Theorem \ref{theoremA} $"(1) \Rightarrow (2)"$, we have used the $d_{1,T^{\mathbb{C}}}$-properness  only on sequences included in $\mathcal{K}(V,\omega_0)^{\mathbb{T}}$. It allows us to obtain the existence of a $(\v, \w)$-cscK metric by the same argument.

The last assertion follows from Theorem \ref{theoremA}.
\end{proof}

\section{Applications}{\label{section-application}}

\subsection{Almost K\"ahler metrics}{\label{subsection-toric}}

As observed in \cite{SKD}, for fixed angular coordinates $dt_0$ with respect to a reference K\"ahler structure $J_0$, one can use (\ref{metric-toric2}) to define a $\mathbb{T}$-invariant \textit{almost-K\"ahler} metric on $V$,  as soon as $\textbf{H}$ satisfies the smoothness, boundary value and positivity conditions  of Proposition \ref{boudary}, even if the inverse matrix $\textbf{G}:= \textbf{H}^{-1}$ is not necessarily the Hessian of a smooth function. We shall refer to such almost K\"ahler metric as \textit{involutive}. One can further use such involutive AK metrics on $V$ to build a compatible metric $\tilde g_{\textbf{H}}$ on $M$, by the formula 

\begin{equation*}
    \tilde{g}_{\textbf{H}}=\sum_{a =1}^k\big(\langle p_a, m \rangle +c_a\big)g_a + \langle dm, \textbf{G} , d m \rangle + \langle \theta,\textbf{H}, \theta \rangle.
\end{equation*}

It is shown in \cite{VA3} that $\tilde g_{\textbf{H}}$ is extremal AK on $M$ in the sense of \cite{ML} (i.e. the hermitian scalar curvature of $\tilde g_{\textbf{H}}$ is a Killing potential) if and only if $\textbf{H}$ satisfies the equation

\begin{equation}{\label{acscK-equation}}
   - \sum_{i,j}\big(\v H_{ij})_{,ij}=\w,
\end{equation}

 \noindent for $(\v,\w)$ defined in $(\ref{weights})$. We shall more generally consider involutive AK metrics  satisfying the equation $(\ref{acscK-equation})$ weight functions $\v>0$ and $\w$. For such AK metrics we say that $(V,J_\textbf{H},,g_\textbf{H},\omega,\mathbb{T})$ is an \textit{involutive }$(\v,\w)$\textit{-csc almost K\"ahler metric}.

The point of considering involutive $(\v,\w)$-csc almost K\"ahler metrics is that  $(\ref{acscK-equation})$ is a linear indeterminate PDE for the smooth coefficients of $\textbf{H}$ (which would therefore admit  infinitely many solutions if we drop the positivity assumption of $\textbf{H}$), which in some special cases is easier to solve explicitly, as demonstrated in \cite{VA3}. On the other hand,  it was observed in \cite{SKD} (see \cite{VA3} for the weighted case) that the existence of a $(\v,\w)$-csc almost K\"ahler metric implies that $\mathcal{F}_{\v, \w}(f) \ge 0$ with equality iff $f=0$,  and it was conjectured that the existence of a $(\v,\w)$-csc almost K\"ahler metric is equivalent to the existence of a $(\v,\w)$-cscK metric on $(V, \omega,\mathbb{T})$. E. Legendre \cite{EL} observed that the existence of an involutive extremal almost K\"ahler metric implies the stronger yet uniform stability of $P$, and thus confirmed the conjecture in the case where $\v=1$ and $\w=\ell_{\textnormal{ext}}$, see $\S \ref{section-extremal-vector-fields}$. Our additional observation is that the same arguments as in the proof of Proposition $\ref{existence-implies-stable}$ show the following.

\begin{prop}
Let $(V, \omega,\mathbb{T})$ be a toric manifold associated to Delzant polytope $P$. Let $\v \in \mathcal{C}^{\infty}(P,\R_{>0})$ and $\w \in \mathcal{C}^{\infty}(P,\R)$ such that $\w$ satisfies $(\ref{annulation-Futaki})$. Suppose there exists an involutive $(\v,\w)$-csc almost K\"ahler metric on $(V, \omega,\mathbb{T})$, i.e. there exists $\textbf{H}$ satisfying the smoothness, boundary value and positivity conditions of Proposition \ref{boudary} and equation $(\ref{acscK-equation})$. Then $P$ is $(\v,\w)$-uniformly K-stable.
\end{prop}

Combining this result (for the special weights (v, w) associated to a semi-simple  principal torus bundle via $(\ref{weights})$) with Theorem \ref{theorem-B}, we deduce:

\begin{prop}{\label{equivalence-almostcsck}}
Let $(V, \omega,\mathbb{T})$ be a toric manifold associated to Delzant polytope $P$. Let $(\v,\w)$ be weights defined in $(\ref{weights})$. Then the following statements are equivalent:

\begin{enumerate}
    \item there exists a $(\v,\w)$-cscK  metric on $(V,\omega,\mathbb{T})$;
    \item there exists an involutive $(\v,\w)$-csc almost K\"ahler metric on $(V,\omega,\mathbb{T})$;
    \item $P$ is $(\v,\w)$-uniformly K-stable in the sense of Definition \ref{uniform-K-stable}.
\end{enumerate}    
\end{prop}

\subsection{Proof of Corollary \ref{prop-ex}}
Let $(S,J_S)$ be a compact complex curve of genus $\textbf{g}$ and  $\mathcal{L}_i \longrightarrow S$ an holomorphic line bundle, $i=0,1,2$. We consider  $(M,J):=\mathbb{P}(\mathcal{L}_0\oplus \mathcal{L}_1 \oplus \mathcal{L}_2)$.  Since the biholomorphism  class of $M$ is invariant by tensoring $\mathcal{L}_0\oplus \mathcal{L}_1 \oplus \mathcal{L}_2$ with a line bundle, we can suppose without loss of generality that  $(M,J)=P(\mathcal{O}\oplus \mathcal{L}_1\oplus \mathcal{L}_2)$, where $\mathcal{O}\longrightarrow S$ is the trivial line bundle and $\textit{\textbf{p}}_i:=deg\big(\mathcal{L}_i\big) \geq0$, $i=1,2$. Suppose $\textit{\textbf{p}}_1=\textit{\textbf{p}}_2=0$, then by a result of Fujiki \cite{AF} $(M,J)$ admits an extremal metric (see also \cite[Remark 2]{VA3}) in every K\"ahler class. If $\textit{\textbf{p}}_2 = \textit{\textbf{p}}_1 > 0$ or $\textit{\textbf{p}}_2 > \textit{\textbf{p}}_1 = 0$ and $\textbf{g}=0,1$, there exists an extremal metric in every K\"ahler class by \cite[Theorem 6]{VA8}. We then suppose $\textit{\textit{\textbf{p}}}_2>\textit{\textbf{p}}_1>0$.

Suppose $S$ is of genus $\textbf{g}=0$, i.e. $(S,J_S)=\mathbb{CP}^1$. Then $M$ is a toric variety. Using the existence of an extremal almost K\"ahler metric of involutive type compatible with any K\"ahler metric on $M$ (see \cite[Proposition 4]{VA3}) and the Yau--Tian--Donaldson correspondence on toric manifold, it is shown in \cite{EL} that there exists an extremal K\"ahler metric in every K\"ahler class of $\mathbb{P}(\mathcal{O}\oplus \mathcal{L}_1 \oplus \mathcal{L}_2) \longrightarrow \mathbb{CP}^1$. Observe that by applying Proposition \ref{equivalence-almostcsck} and Theorem \ref{theoremA} we obtain that these extremal metrics are given by the Calabi ansatz of \cite{VA2}.

Now suppose that $(S,J_S)$ in an elliptic curve, i.e. $\textbf{g}=1$. The complex  manifold $(M,J)$ is not toric. However, it is shown in \cite{VA3} that $(M,J)$ is a semi-simple principal toric fibration. By the Leray-Hirch Theorem $H^2(M,\R)$ is of dimension $2$. In particular, up to scaling any K\"ahler class on $(M,J)$ is compatible. It is shown in \cite[Proposition 4]{VA8}, that $M$ admits an extremal almost K\"ahler metric in any compatible K\"ahler class. Then, using Proposition \ref{equivalence-almostcsck} and Theorem \ref{theorem-B}, we conclude that there exists an extremal K\"ahler metric in every K\"ahler class. Furthermore, by Theorem \ref{theoremA},  the extremal K\"ahler metrics are of the form of (\ref{metriccalabidata}), i.e. are given by the generalized Calabi ansatz.

\begin{remark}
We conclude by pointing out that if $\textnormal{\textbf{g}} \geq 2$, it is shown in \cite[Theorem 2]{VA3} that there exists an extremal K\"ahler metric in sufficiently \textit{small} compatible K\"ahler classes. On the other hand, by \cite[Proposition 2]{VA3}, if $\textnormal{\textbf{g}} > 2$ and $\textit{\textbf{p}}_1,\textit{\textbf{p}}_2$ satisfying $2(\textnormal{\textbf{g}}-1)>\textit{\textbf{p}}_1+\textit{\textbf{p}}_2$, there is no extremal K\"ahler metric in sufficiently \textit{big} K\"ahler classes.
\end{remark}


\begin{thebibliography}{10}


\bibitem{MA2} M. Abreu, \textit{K\"ahler geometry of toric varieties and extremal metrics}, Internat. J. Math. \textbf{9}
(1998), 641–651.

\bibitem{MA3} M. Abreu, \textit{K\"ahler geometry of toric manifolds in symplectic coordinates}, Fields Institute Comm \textbf{35} (2008) , 1-24.

\bibitem{VA1} V.Apostolov, \textit{The K\"ahler geometry of toric manifolds}, Lecture notes availabe at \url{http://profmath.uqam.ca//~apostolo/notes.html}.

\bibitem{VA4} V. Apostolov, D. Calderbank, and P. Gauduchon, \textit{ Hamiltonian 2-Forms in Kähler Geometry, I General Theory}, J. Differential Geom. \textbf{73}, Number 3 (2006), 359-412.

\bibitem{VA2}V. Apostolov, D. M. J. Calderbank, P. Gauduchon and C. Tønnesen-Friedman, 
\textit{Hamiltonian 2-forms in K\"ahler geometry, II Global classification}, J. Differential Geom. \textbf{68} (2004),
277–345.

\bibitem{VA8} V. Apostolov, D. M. J. Calderbank, P. Gauduchon and C. Tønnesen-Friedman,\textit{ Hamiltonian 2-forms
in K\"ahler geometry, III extremal metrics and stability}, Invent. Math. \textbf{173} (2008), 547–601.



\bibitem{VA7} V. Apostolov, D. M. J. Calderbank, and E. Legendre, \textit{Weighted K-stability of polarized varieties and
extremality of Sasaki manifolds}, arXiv:2012.08628, to appear in Adv. Math.

\bibitem{VA3} V. Apostolov, D. M. J. Calderbank, P. Gauduchon and C. Tønnesen-Friedman, \textit{Extremal K\"ahler metrics on projective bundles over a curve.} Adv. Math. \textbf{227} (2011), 2385–2424.

\bibitem{VA6} V. Apostolov, S. Jubert, A. Lahdili, \textit{Weighted K-stability and coercivity with applications to extremal K\"ahler and Sasaki metrics}, arXiv:2104.09709.

\bibitem{VA9} V. Apostolov, J. Keller, \textit{ Relative K-polystability of projective bundles over a curve}, Trans. Amer. Math. Soc.  \textbf{372} (2019), 233-266.

\bibitem{VA5} V. Apostolov, G. Maschler, \textit{Conformally K\"ahler, Einstein-Maxwell geometry}, J. Eur. Math. Soc. \textbf{21} (2019), 1319-1360.



\bibitem{MA} M. F. Atiyah, \textit{ Convexity and commuting Hamiltonians}, Bull. London Math. Soc. \textbf{14} (1982), 1–15.



\bibitem{BB} R. J. Berman, B. Berndtsson, \textit{ Convexity of the K-energy on the space of K\"ahler metrics}, J. Amer.
Math. Soc,\textbf{ 30} (2017), no 4, 1165-1196.

\bibitem{BBJ} R. J. Berman, S. Boucksom, M. Jonsson, \textit{A variational approach to the Yau-Tian-Donaldson conjecture}, J. Amer. Math. Soc. (2021),  https://doi.org/10.1090/jams/964.

\bibitem{BW} R. J. Berman and D. Witt-Nystr\"om, \textit{Complex optimal transport and the pluripotential theory of
K\"ahler-Ricci solitons}, arXiv:1401.8264.

\bibitem{EC2} E.~Calabi, \textit{ Extremal K{\"a}hler metrics}, Seminar on Differential Geometry,  259--290, Ann. of Math. Stud. \textbf{ 102}, Princeton Univ. Press, Princeton, N.J., 1982 259–290.


\bibitem{EC} E. Calabi, \textit{Extremal K\"ahler metrics, II}, Differential Geometry and Complex Analysis, eds. I. Chavel and H. M. Farkas, Springer Verlag (1985), 95–114.


\bibitem{XC} X. Chen, \textit{The space of K\"ahler metrics}, J. Differential Geom. \textbf{56}(2) (2000), 189-234.


 \bibitem{XC5} X. Chen, \textit{On the existence of constant scalar curvature K\"ahler metric: a new perspective}. Ann. Math. du Québec  \textbf{42} (2017), 169- 189.

\bibitem{XC1} X. Chen, J. Cheng, \textit{On the constant scalar curvature K\"ahler metrics: apriori estimates},  J. Amer. Math. Soc. (2021), https://doi.org/10.1090/jams/967.

\bibitem{XC2} X. Chen, J. Cheng, \textit{On the constant scalar curvature K\"ahler metrics II – Existence results}, J. Amer. Math. Soc. (2021), https://doi.org/10.1090/jams/966.



 \bibitem{CSS} X. Chen, S. K. Donaldson, and S. Sun, \textit{ Kähler–Einstein metrics on Fano manifolds. I: Approximation of metrics with cone singularities.} J. Amer. Math. Soc., \textbf{28} (1), 183–197.
 
 
\bibitem{CSS2} X. Chen, S. K. Donaldson, and S. Sun, \textit{K\"ahler–Einstein metrics on Fano manifolds. II: Limits with cone angle less than }$2\pi$. J. Amer. Math. Soc., \textbf{28} (1), 199–234.
 
 
\bibitem{CSS3} X. Chen, S. K. Donaldson, and S. Sun, \textit{K\"ahler–Einstein metrics on Fano manifolds. III: Limits as cone angle approaches }$2\pi$\textit{ and completion of the main proof.} J. Amer. Math. Soc., \textbf{28} (1), 235–278.


\bibitem{CLS} B. Chen, A. Li, L. Sheng, \textit{Uniform K-stability for extremal metrics
on toric varieties}, J. Differential Equations \textbf{257} (2014), 1487–1500.


\bibitem{CPZ} X. X. Chen, M. Paun, Y. Zeng, \textit{ On deformation of extremal metrics}, arXiv:1506.01290v2.


\bibitem{DT} T. Darvas, \textit{The Mabuchi completion of the space of K\"ahler potentials.} Amer. J. Math. \textbf{139} (2017), no. 5,
1275-1313.

\bibitem{DR} T. Darvas, Y.A Rubinstein, \textit{ Tian’s properness conjectures and Finsler geometry of the space of
K\"ahler metrics}. J. Amer. Math. Soc. \textbf{30} (2017), no. 2, 347-387. 



\bibitem{TDD} T. Delcroix, \textit{The Yau-Tian-Donaldson conjecture for cohomogeneity one manifolds}, 	arXiv:2011.07135.





\bibitem{TD} T.Delzant, \textit{Hamiltoniens périodiques et image convexe de l’application moment}, Bull. Soc.
Math. France \textbf{116} (1988), 315–339.

\bibitem{RD} R. Dervan, \textit{Uniform stability of twisted constant scalar curvature Kähler metrics.}
Int. Math. Res. Not. IMRN, \textbf{15} (2016), 4728–4783.

\bibitem{RDJR} R. Dervan, J. Ross, \textit{K-stability for K\"ahler manifolds}	Math. Res. Lett. \textbf{24} (2017), 689–739.


\bibitem{SKD}  S. K. Donaldson, \textit{Scalar curvature and stability of toric varieties}, J. Differential Geom. \textbf{62}
(2002), 289–349.

\bibitem{SKD2} S. K. Donaldson, \textit{Extremal metrics on toric surfaces: a continuity method}, J. Differential Geom. \textbf {79} (2008), 389-432. 



\bibitem {AF} A. Fujiki, \textit{Remarks on extremal Kähler metrics on ruled manifolds}, Nagoya Math. J. \textbf{126} (1992) 89–101.

\bibitem{FM} A. Futaki, T. Mabuchi,  \textit{Bilinear forms and extremal K\"ahler vector fields associated with K\"ahler classes.}
Math. Ann. \textbf{301} (1) (1995), 199-210.



\bibitem{PG} P. Gauduchon, \textit{Calabi’s extremal K\"ahler metrics: An elementary introduction}, Lecture Notes.

\bibitem{DG2} D. Guan, \textit{Existence of extremal metrics on compact almost homogeneous Kähler manifolds
with two ends}, Trans. Am. Math. Soc. \textbf{347} (1995), 2255–2262.

\bibitem{DG}D. Guan, \textit{On Modified Mabuchi Functional and Mabuchi Moduli Space of Kähler Metrics on Toric Bundles}, Math. Res. Lett. \textbf{6} (5)(1999), 547-555.



\bibitem{VG} V. Guillemin, \textit{K\"ahler structures on toric varieties}, J. Differential Geom. \textbf{40} (1994), 285–
309.

\bibitem{VGSS} V. Guillemin and S. Sternberg, \textit{Convexity properties of the moment mapping}, Invent. Math. \textbf{67}
(1982), 491–513.





 \bibitem{YH} Y. Hashimoto, \textit{Existence of twisted constant scalar curvature K\"ahler metrics with a large twist}, Math. Z. \textbf{292} (3-4), 791–803.

\bibitem{TH} T. Hisamoto, \textit{Stability and coercivity for toric polarizations}, arXiv:1610.07998 v3.

\bibitem{WH} W. He, \textit{On Calabi's Extremal Metrics and Properness},  Trans. Amer. Math. Soc. \textbf{372} (2019), 5595–5619.

\bibitem{ADH} A. D. Hwang, \textit{On existence of K\"ahler metrics with constant scalar curvature}, Osaka J.
Math. \textbf{31} (1994), 561–595.

\bibitem{ADH2} A. D. Hwang and M. A. Singer, \textit{A momentum construction for circle-invariant K\"ahler
metrics}, Trans. Amer. Math. Soc. \textbf{354} (2002), 2285–2325.

\bibitem{EI} E. Inoue, \textit{Constant µ-scalar curvature K\"ahler metric - formulation and foundational results}, to appear in J. Geom. Anal., arXiv:1902.00664.

\bibitem{KS} N. Koiso and Y. Sakane, \textit{ Nonhomogeneous Kähler–Einstein metrics on compact complex
manifolds}, in Curvature and Topology of Riemannian Manifolds (Kataka, 1985), Lecture
Notes in Math. \textbf{1201}, Springer, Berlin, 1986, 165–179.

 \bibitem{AL} A. Lahdili,  \textit{K\"ahler metrics with weighted constant scalar curvature and weighted K-stability}, Proc. London Math. Soc. (3), \textbf{119} (2019), 1065–114.
 
 \bibitem{AL2} A. Lahdili, \textit{ Convexity of the weighted Mabuchi functional and the uniqueness of weighted extremal metrics}, 	arXiv:2007.01345. 


\bibitem{LS}  C. R. LeBrun and S. Simanca,\textit{ On the K\"ahler classes of Extremal Metrics}, Geometry and Global Analysis, First MSJ Intern. Res. Inst. Sendai, Japan, Eds. Kotake, Nishikawa and Schoen, 1993.

\bibitem{EL} E. Legendre, \textit{ A note on extremal toric almost Kähler metrics}, (2018), arXiv:1811.05693.

\bibitem {ML} M. Lejmi, \textit{Extremal almost-K\"ahler metrics}, Int. J. Math. \textbf{21}(12) (2010), 1639-1662.

\bibitem{ELST} E. Lerman and S. Tolman, \textit{Hamiltonian torus actions on symplectic orbifolds and toric
varieties}, Trans. Amer. Math. Soc. \textbf{349} (1997), 4201–4230. 

\bibitem{LLS} A. Li, Z. Lian, L. Sheng, \textit{Some Estimates for a Generalized Abreu’s Equation}, Differential Geom. Appl., \textbf{48} (2016), 87–103.

\bibitem{LLS2} A. Li, Z. Lian, L. Sheng, \textit{Interior Regularity for a generalized Abreu Equation}, Int. J. Math.
\textbf{108} (2017), 775–790.


\bibitem{LLS3} A. Li, L. Sheng, G. Zhao, \textit{ Differential inequalities on homogeneous toric bundles},	Journal of Geometry \textbf{108} (2017),  10.1007/s00022-017-0372-4. 

\bibitem{CL} C. Li, \textit{Geodesic rays and stability in the cscK problem}, accepted by Annales Scientifiques de l'ENS, arXiv:2001.01366v3, 2020.

\bibitem{LTW} C. Li, G. Tian, F. Wang \textit{The uniform version of Yau-Tian-Donaldson conjecture for
singular Fano varieties}, arXiv:1903.01215.

\bibitem{TM1} T. Mabuchi, \textit{K-energy maps integrating Futaki invariants}, Tohoku Math. J. (2) \textbf{38} (1986) no. 4,
575–593.


\bibitem{TM} T.Mabuchi, \textit{Some symplectic geometry on compact K\"ahler manifold (I)}, Osaka J. Math. \textbf{24} (1987), 227-252.

\bibitem{YS} Y. Sakane, \textit{ Examples of compact Einstein Kähler manifolds with positive Ricci tensor},
Osaka J. Math. \textbf{23} (1986), 585–616.


\bibitem{GS} G.W. Schwarz, \textit{Smooth functions invariant under the action of a compact Lie group}, Topology \textbf{14} (1975), 63–68.

\bibitem{ZSD} Z. Sj\"ostr\"om-Dyrefelt, \textit{ K-semistability of cscK manifolds with transcendental cohomology class},
J. Geom. Anal. \textbf{28} (2017), 2927 - 2960.


\bibitem{GGS2} G. Székelyhidi, \textit{Extremal metric and K-stability (Phd Thesis)}, arXiv:math/0611002.

\bibitem{GGS} G. Székelyhidi, \textit{Filtrations and test-configurations.} Math. Ann \textbf{362} (2015), 451–484. With an appendix by Sebastien Boucksom.

\bibitem{GT} G. Tian, \textit{K-Stability and K\"ahler-Einstein Metrics}, Communications on Pure and Applied Mathematics, \textbf{68} (7), 1085–1156. 

\bibitem{CTF} Christina W. Tønnesen-Friedman, \textit{ K\"ahler Yamabe minimizers on minimal ruled surfaces}, Math.
Scand. \textbf{9} (2002), 180–190.





\bibitem{ZZ} B. Zhou and X. Zhu, \textit{Relative K-stability and modified K-energy on toric manifolds}, Adv.
Math. \textbf{219} (2008), 1327–1362.


\end{thebibliography}
\end{document}